\documentclass[11pt]{article}%
\usepackage{amsmath}
\usepackage{amsfonts}
\usepackage{amssymb}
\usepackage{graphicx}
\usepackage{float}
\usepackage{tikz}
\usepackage{tkz-graph}
\usepackage{algorithm}
\usepackage{algpseudocode}
\usepackage{placeins}%
\providecommand{\U}[1]{\protect\rule{.1in}{.1in}}
\newtheorem{theorem}{Theorem}[section]

\newtheorem{condition}[theorem]{Condition}

\newtheorem{corollary}[theorem]{Corollary}

\newtheorem{definition}[theorem]{Definition}
\newtheorem{example}[theorem]{Example}

\newtheorem{lemma}[theorem]{Lemma}

\newtheorem{proposition}[theorem]{Proposition}
\newtheorem{remark}[theorem]{Remark}

\newenvironment{proof}[1][Proof]{\noindent \textbf{#1.} }{\  \rule{0.5em}{0.5em}}
\begin{document}

\title{On positionality of trigger strategies Nash Equilibria in SCAR}
\author{G. Konstantinidis and Ath. Kehagias}
\maketitle

\begin{abstract}
We study the positionality of \emph{trigger strategies} Nash equilibria  $\overline{\sigma}$
for the  $N$-player SCAR games $\Gamma_{N}(G|s_{0},\gamma,\varepsilon)$ (with $N\geq3$).
Our study is exhaustive with respect to types of graphs $G$, initial states $s_{0}$ 
and values of $N,\gamma,\varepsilon$. 
We conclude that in the majority of cases,
profiles $\overline{\sigma}$ are nonpositional. Whenever $\overline{\sigma}$ are positional
a key role is played by paths and the $\varepsilon$, $\gamma$ values
(especially whether $\varepsilon>0$ or not).
A crucial concept in our analysis is the \emph{state cop number}, which is first introduced in the current paper. 

\end{abstract}

\section{Introduction\label{sec01}}

In \cite{Konstantinidis2019} we introduce the game of \emph{selfish cops and
adversarial robber} (SCAR) denoted by $\Gamma_{N}\left(  G|s_{0}%
,\gamma,\varepsilon\right)  $. This can be seen as an $N$-player variant of
the classic, two-player cops and robber (CR)\ game
\cite{Nowakowski1983,Quilliot1985} where each of the $N-1$ cops is controlled
by a different player (whereas in CR a single player controls all cops). $G$
denotes the graph of the game, $s_{0}$ the starting \emph{state }or
\emph{position} and $\gamma,\varepsilon$ are game parameters. In that paper we
prove (among other results) that $\Gamma_{N}\left(  G|s_{0},\gamma
,\varepsilon\right)  $ has a \emph{Nash equilibrium} (NE)\ in \emph{trigger
strategies} \cite{Thuijsman1997} which is generally \emph{nonpositional}
(i.e., some player's\ next move depends on the past).

In the current paper we explore conditions on the form of $G$, $s_{0}$ and the
values of $N,\gamma,\varepsilon$ under which a trigger strategies NE\ of
$\Gamma_{N}\left(  G|s_{0},\gamma,\varepsilon\right)  $ is actually
\emph{positional} (i.e., each player's next move depends only on the current
state). To obtain an exhaustive list of such conditions it is necessary to use
a graph classification based on a novel concept, the \emph{state cop number},
which is a refinement of the classic cop number \cite{Aigner84}.

\section{Preliminaries\label{sec02}}

All graphs considered here are undirected, finite, connected and simple.
$N\left(  u\right)  $ denotes the open neighborhood of vertex $u$, $d(u,v)$
the distance between $u$ and $v$ and $\left\vert A\right\vert $ the
cardinality of set $A$. In classic CR, $c\left(  G\right)  $ denotes the
\emph{cop number of graph} $G$, defined as \emph{the minimum number of cops
needed to ensure capture from any possible starting position}.

In this paper we focus on \emph{auxiliary games} $\Gamma_{N}^{n}\left(
G|s_{0},\gamma,\varepsilon\right)  $ which we use in \cite{Konstantinidis2019}
to prove the existence of a NE\emph{ }in $\Gamma_{N}\left(  G|s_{0}%
,\gamma,\varepsilon\right)  $. The full story can be found in
\cite{Konstantinidis2019}; here we present basic notation and facts.

For each $n\in\left\{  1,...,N\right\}  $, $\Gamma_{N}^{n}\left(
G|s_{0},\gamma,\varepsilon\right)  $ is a zero-sum two-player \emph{stochastic
game} \cite{FilarVrieze2012} played on the graph $G=\left(  V,E\right)  $;
$\gamma\in\left(  0,1\right)  $ and $\varepsilon\in\left[  0,\frac{1}%
{N-1}\right]  $ are \emph{game parameters}. Player $P_{n}$ controls the $n$-th
token, i.e., cop $C_{n}$ for $n\in\left\{  1,...,N-1\right\}  $, or the robber
$R$ for $n=N$; player $P_{-n}$ controls the rest tokens.

A \emph{state} $s$ of the game has the form $s=\left(  x^{1},...,x^{N}%
,n\right)  $ where $x^{i}\in V$ is the location of the $i$-th token and
$n\in\left\{  1,...,N\right\}  $ is the token moving next. $S$ denotes the set
of all states and can be partitioned as:

\begin{enumerate}
\item $S=S_{c}\cup S_{nc}\cup\{\tau\}$ where, $S_{c}$ is the set of
\emph{capture states,} in which at least one $C_{n}$ is on the same vertex as
$R$ (i.e., $x^{n}=x^{N}$), $S_{nc}$ the set of \emph{noncapture states} where
no $C_{n}$ is on the same vertex as $R$ and $\tau$ the terminal state that
finally occurs in case of capture; or

\item $S=\cup_{n=1}^{N}S^{n}\cup\{\tau\}$ where $S^{n}:=\left\{  s:s=\left(
x^{1},...,x^{N},n\right)  \right\}  $ is the set of states in which the $n$-th
token has the move.
\end{enumerate}

The game begins at some \emph{initial state} $s_{0}$, which also specifies the
first token to move. In each turn a single token is moved from its current
vertex to a neighboring one, always following the sequence $...,C_{1}%
,C_{2},...,C_{N-1},R,C_{1},...$ . The game lasts an infinite number of turns
but is effectively over as soon as capture occurs, if it does, since right
after the system moves to state $\tau$ and stays there ad infinitum.

Regarding the players' \emph{payoffs} and since $\Gamma_{N}^{n}\left(
G|s_{0},\gamma,\varepsilon\right)  $ is a zero-sum game, it suffices to
specify $P_{n}$'s payoff. If the \emph{capture time} (i.e., the first time
that a capture state $s\in S_{c}$ occurs) is $t$, then:

\begin{enumerate}
\item in the game $\Gamma_{N}^{N}\left(  G|s_{0},\gamma,\varepsilon\right)  $,
$P_{N}$'s payoff is $-\gamma^{t}$;

\item in the game $\Gamma_{N}^{n}\left(  G|s_{0},\gamma,\varepsilon\right)  $
($n\in\left\{  1,...,N-1\right\}  $), $P_{n}$'s payoff is
\[%
\begin{array}
[c]{cl}%
\dfrac{\left(  1-\varepsilon\right)  }{K}\gamma^{t} & :\text{when
}K\ \text{cops (}K\in\left\{  1,...,N-2\right\}  \text{) \emph{including}
}C_{n}\text{\ are in the same vertex as }R\text{; }\\
\dfrac{\varepsilon}{N-K-1}\gamma^{t} & :\text{when }K\text{ cops (}%
K\in\left\{  1,...,N-2\right\}  \text{)\ \emph{but no} }C_{n}\text{ are in the
same vertex as }R\text{;}\\
\dfrac{\gamma^{t}}{N-1} & :\text{when all }N-1\text{ cops are in the same
vertex as }R\text{.}%
\end{array}
\]

\end{enumerate}

\noindent In case capture never takes place, $P_{n}$'s payoff is zero.
$P_{-n}$'s payoff is always the negative of $P_{n}$'s.

$\Gamma_{N}^{N}\left(  G|s_{0},\gamma,\varepsilon\right)  $ is essentially the
\emph{modified} CR\ game introduced in \cite{Konstantinidis2019}, the basic
difference with classic CR being that time is counted in turns instead of
rounds; a single token, cop or robber, moves in each turn. Hence $P_{-N}$ aims
to effect capture as soon as possible and $P_{N}$ aims to delay it as much as
possible. In this case the players' optimal strategies depend on $G$ and
$s_{0}$, but not on $\gamma$, $\varepsilon$.

In games $\Gamma_{N}^{n}\left(  G|s_{0},\gamma,\varepsilon\right)  $ however,
where $C_{n}$ plays against the robber and $N-2$ \textquotedblleft
robber-friendly\textquotedblright\ cops, the players' optimal strategies
depend also on the values of $\gamma,\varepsilon$. If $\varepsilon>0$, then
$P_{-n}$'s best outcome is evasion of the robber and from then on, depending
on $\gamma$, $\varepsilon$ and the respective capture times he may prefer a
capture by one of his tokens, or a \emph{joint} capture involving also $C_{n}%
$, or a late as possible \emph{pure} $C_{n}$ capture. The difference in case
$\varepsilon=0$ is that $P_{-n}$ is indifferent between letting the robber
evade and capturing him by one of his tokens, since in both cases he gets his
best outcome, i.e., a zero payoff.

We will shortly return to these optimal strategies. But first let us introduce
a few additional notions.

A \emph{history} $h=(s_{0},s_{1},...)$ is a finite or infinite sequence of
states. Let $H_{f}$ denote the set of \emph{finite} length histories
$(s_{0},...,s)$, $H_{f}^{n}$ those where token $n$ moves, i.e., $s\in S^{n}$
and $H_{fnc}^{n}$ those where $s\in S^{n}\cap S_{nc}$.

A \emph{pure}\ (or deterministic) strategy for the $m$-th token is a function
$\sigma^{m}$ which maps finite histories to next$\ $moves. That is, $\forall
h=(s_{0},s_{1},...,s_{t})\in H_{f}:\sigma^{m}\left(  h\right)  =v$ specifies
that: if the game started at $s_{0}$ and passed through $s_{1},...,s_{t}$,
then next the $m$-th token should move to vertex $v$.

A strategy is \emph{positional}\ (or \emph{stationary Markovian}) if it
depends only on the current state $s_{t}$, i.e.,
\[
\forall h=(s_{0},...,s_{t})\in H_{f}\text{, }\sigma^{m}\left(  h\right)
=\sigma^{m}\left(  s_{t}\right)  .
\]

A \emph{strategy profile} (or, simply a profile) is a tuple $\sigma=\left(
\sigma^{1},...,\sigma^{N}\right)  $. A profile is \emph{positional} if it
consists solely of positional strategies. Otherwise it is called
\emph{nonpositional}.

Standard results \cite{FilarVrieze2012} yield that the players in games
$\Gamma_{N}^{n}\left(  G|s_{0},\gamma,\varepsilon\right)  $ have optimal,
\emph{pure positional} strategies for all $n$, $s_{0}$ and $\gamma
,\varepsilon$. This is the only kind of strategies we consider then for these
games. Furthermore, since $P_{n}$ controls only the $n$-th token, his strategy
consists of a single function $\sigma^{n}$; a $P_{-n}$ strategy though is a
\textquotedblleft vector\textquotedblright\ function $\sigma^{-n}=\left(
\sigma^{m}\right)  _{m\in\left\{  1,...,N\right\}  \backslash n}$ with one
strategy for each of $P_{-n}$'s tokens; and both players' strategies together
yield profile $\sigma=\left(  \sigma^{1},...,\sigma^{N}\right)  =\left(
\sigma^{n},\sigma^{-n}\right)  $.

Since the form of the games is deterministic, if moreover the players use pure
strategies, the games \emph{evolve deterministically}. That is, given an
initial state $s$ and a pure strategy profile $\sigma$, the tuple $(s,\sigma)$
leads in a deterministic manner to either capture or evasion of the robber.

Keeping the above in mind, in the sequel, we will need the following definitions.

$T\left(  s,\sigma\right)  $ denotes the \emph{capture time} (finite or not)
starting from state $s$ under pure profile $\sigma$.

$\widehat{\sigma}^{n}$ denotes a pure positional \emph{optimal strategy}
\emph{for the} $n$-th \emph{token in modified CR game} $\Gamma_{N}^{N}\left(
G|s_{0},\gamma,\varepsilon\right)  $, and $\widehat{\Sigma}^{n}$ the set of
all $\widehat{\sigma}^{n}$'s. Strategies $\widehat{\sigma}^{n}$ will be called
\emph{CR-optimal}.

Note that: for \emph{any} initial state $s$ and CR-optimal profile
$\widehat{\sigma}=(\widehat{\sigma}^{1},...,\widehat{\sigma}^{N})$ such that
$(s,\widehat{\sigma})$ leads to capture, \emph{it is always the same cop
effecting capture and at the same time }$T\left(  s,\widehat{\sigma}\right)
$. This follows from the facts:\ (i)\ in each turn only one token moves and
(ii)\ under CR-optimal play, $R$ does \emph{never} run into a
cop\footnote{Contrary to the classic CR, where capture under optimal play
occurs in the minimum number of \emph{rounds }and it can possibly be effected
by \emph{different} cops, in modified CR capture under optimal play (i.e., for
every $\widehat{\sigma}\in\widehat{\Sigma}=\times_{n\in\{1,...,N\}}%
\widehat{\Sigma}^{n}$) occurs in the minimum number of \emph{moves} and thus
always by the \emph{same} cop.}. Let $\widehat{C}(s)$ then denote the
\emph{cop effecting capture under every CR-optimal profile} $\widehat{\sigma}%
$, \emph{when the game starts at} $s$, and $\widehat{T}(s):=T\left(
s,\widehat{\sigma}\right)  $ the respective \emph{number of} \emph{moves. }If
$\widehat{C}(s)=C_{m}$ for some $m\in\{1,...,N-1\}$, then at times we denote
this by $\widehat{C}_{m}(s)$.

Finally, for all $n,m\in\{1,...,N\}$, let $\phi_{m}^{n}$ be an \emph{optimal,
pure positional strategy for the} $n$-\emph{th token in game} $\Gamma_{N}%
^{m}\left(  G|s_{0},\gamma,\varepsilon\right)  $ (where dependence on
$\gamma,\varepsilon$ has been supressed); let $\Phi_{m}^{n}$ be the respective set.

We now turn to SCAR games $\Gamma_{N}\left(  G|s_{0},\gamma,\varepsilon
\right)  $ and the so-called \emph{trigger strategies} profiles $\overline
{\sigma}$, where%
\[
\overline{\sigma}:=\left(  \overline{\sigma}^{1},...,\overline{\sigma}%
^{N}\right)
\]
and each trigger strategy $\overline{\sigma}^{n}$ is composed from strategies
$\phi_{m}^{n}$ as follows:

For all $h=(s_{0},...,s)\in H_{f}$%
\[
\overline{\sigma}^{n}(h):=\left\{
\begin{array}
[c]{ll}%
\phi_{n}^{n}(s) & \text{as long as every player }m\in\left\{  1,...,N\right\}
\backslash n\text{ follows }\phi_{m}^{m}\text{; }\\
\phi_{m}^{n}(s) & \text{as soon as some player }m\in\left\{  1,...,N\right\}
\backslash n\ \text{ deviates from }\phi_{m}^{m}\text{.}%
\end{array}
\right.
\]
In general, $\overline{\sigma}^{n}$ is nonpositional by construction (and so
is $\overline{\sigma}$ then) since it takes into account the players' past
behavior. However, if $\phi_{n}^{n}\left(  s\right)  =\phi_{m}^{n}\left(
s\right)  $ for all $m$ and \textquotedblleft relevant\textquotedblright%
\ states $s$, $\overline{\sigma}^{n}$ becomes positional.

To better understand the meaning of this latter condition, consider the
following. Roughly speaking, cop $C_{n}$'s optimal strategy $\phi_{n}^{n}$ in
the game $\Gamma_{N}^{n}$ (where he plays against a \textquotedblleft
coalition\textquotedblright\ of the remaining players) must be also optimal
(i)\ in every game $\Gamma_{N}^{m}$ with $m\in\left\{  1,...,N-1\right\}
\backslash n$ (where he and the remaining players ally against cop $C_{m}$)
and (ii)\ in the CR game $\Gamma_{N}^{N}$ (where he and the remaining cops
chase the robber). In other words: $\overline{\sigma}$ is positional iff, for
every $m$ and $n$, the $n$-th token's CR-optimal strategy $\widehat{\sigma
}^{n}$ is also optimal in $\Gamma_{N}^{m}$ . This can be stated formally as follows.

\begin{condition}
\label{conpos1}Let $\overline{\sigma}=(\overline{\sigma}^{1},...,\overline
{\sigma}^{N})$ be a trigger strategies profile in $\Gamma_{N}(G|s_{0}%
,\gamma,\varepsilon)$ with $\overline{\sigma}^{n}$ consisting of $\left(
\phi_{m}^{n}\right)  _{m=1}^{N}$. Then $\overline{\sigma}$ is positional
(resp. nonpositional) iff \textbf{A1} (resp. \textbf{A2}) holds:%
\begin{align}
\mathbf{A1}  &  :\forall n,m\in\left\{  1,...,N\right\}  \text{, }%
\exists\widehat{\sigma}^{n}\in\widehat{\Sigma}^{n}:\forall h=\left(
s_{0},...,s\right)  \in H_{fnc}^{n}\text{, }\phi_{m}^{n}(s)=\widehat{\sigma
}^{n}(s),\\
\mathbf{A2}  &  :\exists n,m\in\left\{  1,...,N\right\}  :\forall
\widehat{\sigma}^{n}\in\widehat{\Sigma}^{n}\text{, }\exists h=\left(
s_{0},...,s\right)  \in H_{fnc}^{n}:\phi_{m}^{n}(s)\neq\widehat{\sigma}%
^{n}(s).
\end{align}

\end{condition}

The goal of this paper is to explore the form of graphs $G$ and initial states
$s_{0}$ and the values of parameters $N$, $\gamma$ and $\varepsilon$, under
which, each of the mutually exclusive conditions \textbf{A1} or \textbf{A2}
holds. Given $\Gamma_{N}^{N}$ is the modified CR, $\Phi_{N}^{n}=\widehat
{\Sigma}^{n}$ and $\mathbf{A1}$ holds always for $m=N$. Hence we will be
examining the remaining cases.

\section{Analysis\label{sec03}}

The study is divided as follows. In Section \ref{sec0401} we study the case
$\left\vert V\right\vert =2,\varepsilon\geq0$ and in Section \ref{sec0402} the
case $\left\vert V\right\vert >2,\varepsilon>0$; Section \ref{sec0403}
concerns the case $\left\vert V\right\vert >2,\varepsilon=0$ and is further
divided in Section \ref{sec040301}, where $c(G)\leq N-1$, and in Section
\ref{sec040302}, where $c(G)>N-1$.

To avoid trivialities we only consider initial states $s_{0}\in S_{nc}$.

\subsection{Case $\left\vert V\right\vert =2$ (the path $\mathcal{P}_{2}$)
$\varepsilon\geq0$\label{sec0401}}

\begin{proposition}
\label{PropposP2}Let $G=\left(  V,E\right)  $ with $\left\vert V\right\vert
=2$. Then $\Gamma_{3}(G|s_{0},\gamma,\varepsilon)$ has a (unique) positional
trigger strategies profile $\overline{\sigma}$ iff $\varepsilon\in\left[
0,\frac{1}{2}\right)  $ and $\gamma\in\left[  \sqrt{\frac{\varepsilon
}{1-\varepsilon}},\frac{1}{2-2\varepsilon}\right]  $.
\end{proposition}

\begin{proof}
The set of noncapture states is
\[
S_{nc}=\left\{
(1,1,2,1),(1,1,2,2),(1,1,2,3),(2,2,1,1),(2,2,1,2),(2,2,1,3)\right\}  \text{.}%
\]
Due to symmetry we only consider initial states $s_{0}=(1,1,2,n)$,
$n\in\left\{  1,2,3\right\}  $; then each $\left(  1,1,2,n\right)  $ has two
possible successors (e.g., the successors of $\left(  1,1,2,1\right)  $ are
$\left(  1,1,2,1\right)  $ and $\left(  1,1,2,2\right)  $) i.e., given $s_{0}%
$, each token has two positional strategies. Thus, in this case Condition
\textbf{A1} becomes:%
\begin{equation}
\forall n,m\in\left\{  1,2,3\right\}  ,\exists\widehat{\sigma}^{n}\in
\widehat{\Sigma}^{n}:\phi_{m}^{n}(1,1,2,n)=\widehat{\sigma}^{n}%
(1,1,2,n)\text{.} \label{eq001}%
\end{equation}

We examine under which conditions (\ref{eq001}) holds and hence $\overline
{\sigma}=(\overline{\sigma}^{1},\overline{\sigma}^{2},\overline{\sigma}^{3})$
is positional.

To begin with, the (unique) CR-optimal strategies are: for the cops,
$\widehat{\sigma}^{1}(1,1,2,1)=\widehat{\sigma}^{2}(1,1,2,2)=2$ (immediate
capture), and$\ $for the robber $\widehat{\sigma}^{3}(1,1,2,3)=2$ (stay in place).

\medskip

\noindent\underline{I. In game $\Gamma_{3}^{1}(G|s_{0},\gamma,\varepsilon)$}
($P_{1}$ controls $C_{1}$, $P_{-1}$ controls $C_{2}$ and $R$) we have the following.

For token $C_{1}$, the unique optimal strategy $\phi_{1}^{1}\in\Phi_{1}^{1}$
prescribes immediate capture at $s=(1,1,2,1)$, i.e., $\phi_{1}^{1}(s)=2$.
Indeed, if $C_{1}$ captures say at time $t$, then $P_{1}$'s payoff is
$(1-\varepsilon)\gamma^{t}$. Otherwise, depending on the values of
$\gamma,\varepsilon$, $P_{-1}$ will optimally play so that, either $C_{2}$
effects capture in the next move, and then $P_{1}$'s payoff will be
$\varepsilon\gamma^{t+1}$, or $C_{2}$ stays put and $R$ runs into both
$C_{1},C_{2}$ in the next move, and $P_{1}$'s payoff will be $\frac{1}%
{2}\gamma^{t+2}$. But $\left(  1-\varepsilon\right)  \gamma^{t}>\max\left(
\varepsilon\gamma^{t+1},\frac{1}{2}\gamma^{t+2}\right)  $ for all $\left(
\gamma,\varepsilon\right)  \in\left(  0,1\right)  \times\lbrack0,\frac{1}{2}%
]$. Hence,%
\begin{equation}
\forall\left(  \gamma,\varepsilon\right)  \in\left(  0,1\right)  \times
\lbrack0,\frac{1}{2}]\text{, }\phi_{1}^{1}(1,1,2,1)=2=\widehat{\sigma}%
^{1}(1,1,2,1)\text{.} \label{eq003a}%
\end{equation}

For token $C_{2}$, moving at time $t$ from state $s=(1,1,2,2)$ we have the
following possibilities.

\begin{enumerate}
\item $C_{2}$ captures at $t$; $P_{-1}$'s loss is $\varepsilon\gamma^{t}$;

\item $C_{2}$ stays put at $t$, $R$ runs into both $C_{1},C_{2}$ at $t+1$;
$P_{-1}$'s loss is $\frac{1}{2}\gamma^{t+1}$;

\item $C_{2}$ and $R$ stay put and $C_{1}$ captures at $t+2$; $P_{-1}$'s loss
is $(1-\varepsilon)\gamma^{t+2}$.
\end{enumerate}

\noindent For a positional $\overline{\sigma}^{2}$, $P_{-1}$ must not prefer
(2) or (3) to (1), i.e.,%
\begin{equation}
\varepsilon\gamma^{t}\leq\min\left(  \frac{1}{2}\gamma^{t+1}\text{
},(1-\varepsilon)\gamma^{t+2}\right)  \Rightarrow\gamma\geq\max\left(
2\varepsilon,\sqrt{\frac{\varepsilon}{1-\varepsilon}}\right)  \label{eq003b}%
\end{equation}
Given $\gamma<1$, from $\gamma\geq2\varepsilon$ we get $\varepsilon<\frac
{1}{2}$. Therefore,
\begin{equation}
\varepsilon\in\left[  0,\frac{1}{2}\right)  \text{ and }\gamma\in\left[
\sqrt{\frac{\varepsilon}{1-\varepsilon}},1\right)  \label{eq003d}%
\end{equation}
guarantee the existence of a $\phi_{1}^{2}$ such that:
\[
\phi_{1}^{2}(s)=\widehat{\sigma}^{2}(s)\text{.}%
\]

For token $R$, moving at time $t$ from state $s=(1,1,2,3)$ we have the
following possibilities.

\begin{enumerate}
\item $R$ stays put at $t$ and $C_{1}$ effects capture at $t+1$; $P_{-1}$'s
loss is $(1-\varepsilon)\gamma^{+1}$.

\item $R$ runs into both $C_{1},C_{2}$ at $t$; $P_{-1}$'s loss is $\frac{1}%
{2}\gamma^{t}$.
\end{enumerate}

\noindent For a positional $\overline{\sigma}^{3}$, $P_{-1}$ must not prefer
(2) to (1):
\begin{equation}
(1-\varepsilon)\gamma^{t+1}\leq\frac{1}{2}\gamma^{t}\Rightarrow\gamma\leq
\frac{1}{2-2\varepsilon}\text{.} \label{eq003f}%
\end{equation}

\noindent\underline{II. In game $\Gamma_{3}^{2}(G|s_{0},\gamma,\varepsilon)$}
($P_{2}$ controls $C_{2}$, $P_{-2}$ controls $C_{1}$ and $R$) we have the following.

Regarding token $C_{2}$:\ moving at $t$ from $s=(1,1,2,2)$, the unique optimal
strategy prescribes immediate capture, i.e., $\phi_{2}^{2}(s)=2=\widehat
{\sigma}^{2}(s)$, for all $\left(  \gamma,\varepsilon\right)  \in\left(
0,1\right)  \times\lbrack0,\frac{1}{2}]$. Indeed, $P_{2}$'s payoff in this
case is $(1-\varepsilon)\gamma^{t}$. Otherwise, optimally $R$ stays in place
at $t+1$ and at $t+2$ $C_{1}$ captures with $P_{2}$'s payoff being
$\varepsilon\gamma^{t+2}<(1-\varepsilon)\gamma^{t}$.

Regarding token $C_{1}$:

\begin{enumerate}
\item if $C_{1}$ captures at $t$, then $P_{-2}$'s loss is $\varepsilon
\gamma^{t}$;

\item otherwise (optimally) $C_{2}$ captures at $t+1$ and $P_{-2}$'s loss is
$(1-\varepsilon)\gamma^{t+1}$.
\end{enumerate}

\noindent For a positional $\overline{\sigma}^{1}$ it must be:
\begin{equation}
\varepsilon\gamma^{t}\leq(1-\varepsilon)\gamma^{t+1}\Rightarrow\gamma\geq
\frac{\varepsilon}{1-\varepsilon}\text{.} \label{eq004}%
\end{equation}

Regarding the robber token $R$ we have the following possibilities.

\begin{enumerate}
\item $R$ stays put at time $t$ and $C_{1}$ effects capture at $t+1$; $P_{-2}%
$'s loss is $\varepsilon\gamma^{t+1}$;

\item $R$ and $C_{1}$ stay put and $C_{2}$ captures at time $t+2$; $P_{-2}$'s
loss is $(1-\varepsilon)\gamma^{t+2}$.

\item $R$ runs into both $C_{1},C_{2}$ at $t$; $P_{-2}$'s loss is $\frac{1}%
{2}\gamma^{t}$.
\end{enumerate}

\noindent\noindent For a positional $\overline{\sigma}^{3}$ (1)\ must be at
least as good for $P_{-2}$ as (2) and (3). By (\ref{eq004}) we have,
\[
\gamma\geq\frac{\varepsilon}{1-\varepsilon}\Rightarrow\varepsilon\gamma
^{t+1}\leq(1-\varepsilon)\gamma^{t+2}%
\]
and $P_{-2}$ does not prefer (2) to (1). \ For (1)\ to be at least as good as
(3) it must be $\varepsilon\gamma^{t+1}\leq\frac{1}{2}\gamma^{t}$. If
$\varepsilon=0$, this holds always. If $\varepsilon>0$ it must be%
\begin{equation}
\gamma\leq\frac{1}{2\varepsilon}\text{.} \label{eq010a}%
\end{equation}

For a positional profile $\overline{\sigma}=\left(  \overline{\sigma}%
^{1},\overline{\sigma}^{2},\overline{\sigma}^{3}\right)  $, (\ref{eq003d}%
)-(\ref{eq010a}) must all hold; this yields the required result.\bigskip
\end{proof}

\begin{proposition}
\label{prop0302}Let $G=\left(  V,E\right)  $ with $\left\vert V\right\vert
=2$; let $N>3$ and $\varepsilon>0$. Then every trigger strategies profile
$\overline{\sigma}$ of $\Gamma_{N}(G|s_{0},\gamma,\varepsilon)$ is
nonpositional, for all $\gamma\in(0,1)$.
\end{proposition}

\begin{proof}
Consider game $\Gamma_{N}^{1}(G|s_{0},\gamma,\varepsilon)$ ($P_{1}$ controls
$C_{1}$, $P_{-1}$ controls $C_{2},C_{3},...,C_{N-1}$ and $R$). Let
$s=(1,...,1,2,2)$ and consider $C_{2}$'s optimal move at time $t$ from $s$. If
$C_{2}$ captures $R$, then $P_{-1}$'s loss is $\varepsilon\gamma^{t}$; if
$C_{2}$ stays put and $C_{3}$ captures at $t+1$, $P_{-1}$'s loss is
$\varepsilon\gamma^{t+1}<\varepsilon\gamma^{t}$ (since $\varepsilon>0$). So
for the unique $\phi_{1}^{2},\widehat{\sigma}^{2}$ it is $\phi_{1}^{2}\left(
s\right)  =1\neq2=\widehat{\sigma}^{2}\left(  s\right)  $. We conclude that
$\overline{\sigma}^{2}$ and hence $\overline{\sigma}$ is always nonpositional.
\end{proof}

\begin{proposition}
\label{prop0303a}Let $G=\left(  V,E\right)  $ with $\left\vert V\right\vert
=2$; let $N>3$ and $\varepsilon=0$. The following hold.

\begin{enumerate}
\item $\Gamma_{N}(G|s_{0},\gamma,\varepsilon)$ has at least one nonpositional
trigger strategies profile $\overline{\sigma}$, for all $\gamma\in(0,1).$

\item $\Gamma_{N}(G|s_{0},\gamma,\varepsilon)$ has at least one positional
trigger strategies profile $\overline{\sigma}$ iff $\gamma\in\left(
0,\frac{1}{N-1}\right]  $.
\end{enumerate}
\end{proposition}

\begin{proof}
All $s\in S_{nc}$ are of the form $s=\left(  v_{1},...,v_{1},v_{2},n\right)  $
with $v_{1}\neq v_{2}$ (no cop is in the same vertex as the robber). At any
such state $s$ the unique CR-optimal strategies are: for cop $C_{n}$,
$\widehat{\sigma}^{n}(v_{1},...,v_{1},v_{2},n)=v_{2}$ (immediate
capture),$\ $for the robber $R$, $\widehat{\sigma}^{N}(v_{1},...,v_{1}%
,v_{2},N)=v_{2}$ (stay in place).

\smallskip

\noindent1. Consider game $\Gamma_{N}^{1}(G|s_{0},\gamma,\varepsilon)$ and
token $C_{2}$ having the next move at state $s=\left(  v_{1},...,v_{1}%
,v_{2},2\right)  $. Moving to $v_{2}$ effects a capture which gives $P_{-1}$
his minimum loss of 0; but so does staying in place, provided some other
$C_{k}$ ($k\in\left\{  3,...,N-1\right\}  $) effects the capture. Thus there
exists $\phi_{1}^{2}:$ $\phi_{1}^{2}\left(  s\right)  =v_{1}\neq
v_{2}=\widehat{\sigma}^{2}\left(  s\right)  $. Hence there exists
$\overline{\sigma}^{2}$ and thus $\overline{\sigma}$ which is nonpositional.

\smallskip

\noindent2. Consider game $\Gamma_{N}^{m}(G|s_{0},\gamma,\varepsilon)$
($m\in\left\{  1,2,...,N-1\right\}  $)

For $C_{m}$: immediate capture (i.e., moving to $v_{2}$) results to a payoff
of $(1-\varepsilon)\gamma^{t}=\gamma^{t}>0$ for $P_{m}$. If $C_{m}$ does not
capture immediately, then optimally $P_{-m}$ effects a $C_{n}$ ($n\neq m$)
capture any time before $C_{m}$ resulting to a payoff of $0<\gamma^{t}$ for
$P_{m}$. So $P_{m}$ prefers immediate capture and thus for the only $\phi
_{m}^{m}$ it is $\phi_{m}^{m}\left(  v_{1},...,v_{1},v_{2},m\right)
=v_{2}=\widehat{\sigma}^{m}\left(  v_{1},...,v_{1},v_{2},m\right)  $.

For $C_{n}$ ($n\in\left\{  1,...,N-1\right\}  \backslash m$): moving to
$v_{2}$ effects a capture which gives $P_{-m}$ his minimum loss of 0. Thus
there exists $\phi_{m}^{n}:\phi_{m}^{n}\left(  v_{1},...,v_{1},v_{2},n\right)
=v_{2}=\widehat{\sigma}^{n}\left(  v_{1},...,v_{1},v_{2},n\right)  $.

For $R$: If $m\in\left\{  2,...,N-1\right\}  $, then $R$ optimally stays put
and capture is effected by any $C_{n}$ with $n<m$, resulting to a loss of $0$
for $P_{-m}$. For $m=1$ we have the following. If $R$ stays put at time $t$,
$C_{1}$ captures at $t+1$ and $P_{-m}$'s loss is $(1-\varepsilon)\gamma
^{t+1}=\gamma^{t+1}$. If $R$ runs into\ (all) cops, $P_{-m}$'s loss is
$\frac{1}{N-1}\gamma^{t}$. Thus if $\gamma^{t+1}\leq\frac{\gamma^{t}}%
{N-1}\Leftrightarrow\gamma\leq\frac{1}{N-1}$, there exists $\phi_{m}^{N}%
:\phi_{m}^{N}\left(  v_{1},...,v_{1},v_{2},N\right)  =v_{2}=\widehat{\sigma
}^{N}\left(  v_{1},...,v_{1},v_{2},N\right)  $. Hence, \textbf{A1} holds and a
positional profile $\overline{\sigma}$ exists iff $\gamma\in\left(  0,\frac
{1}{N-1}\right]  $.
\end{proof}

\subsection{Case $\left\vert V\right\vert \geq3$ and $\varepsilon
>0$\label{sec0402}}

From this point on, unless otherwise specified, we consider $N\geq3$.

\begin{proposition}
\label{prop0303}Let $G=\left(  V,E\right)  $ with $\left\vert V\right\vert
\geq3$; let $\varepsilon>0$. Then every trigger strategies profile
$\overline{\sigma}$ of $\Gamma_{N}(G|s_{0},\gamma,\varepsilon)$ is
nonpositional, for all $\gamma\in(0,1)$.
\end{proposition}

\begin{proof}
Given $\left\vert V\right\vert \geq3$, from any $s_{0}$ it is possible to
reach at some $t$\ a state $s=\left(  v_{1},v_{2},...,v_{N},2\right)  \in
S_{nc}$ (so $C_{2}$ moves next) with $v_{1}\neq v_{2}$ and $v_{N}\in N\left(
v_{2}\right)  $. Now, at $s$, every CR-optimal $\widehat{\sigma}^{2}$ dictates
immediate capture by $C_{2}$, i.e., $\forall\widehat{\sigma}^{2}\in
\widehat{\Sigma}^{2}$, $\widehat{\sigma}^{2}(s)=v_{N}$. In game $\Gamma
_{N}^{1}(G|s_{0},\gamma,\varepsilon)$ though at state $s$, if $C_{2}$ captures
immediately $R$, $P_{-1}$'s loss is $\varepsilon\gamma^{t}$. If $P_{-1}$ keeps
$C_{2},...,C_{N-1}$ in place and lets $R$ move to $v_{2}$ at $t+N-2$, then, if
$k\in\{1,...,N-2\}$ is the number of $P_{-1}$'s cop tokens located at $v_{2}$
including $C_{2}$, $P_{-1}$'s loss is $\frac{\varepsilon}{N-k-1}\gamma
^{t+N-2}<\varepsilon\gamma^{t}$. Hence, at $s$, $P_{-1}$ prefers to defer
capture and \emph{never} capture with $C_{2}$. Thus,
\[
\forall\phi_{1}^{2}\in\Phi_{1}^{,2},\forall\widehat{\sigma}^{2}\in
\widehat{\Sigma}^{2}\text{, }\exists h=\left(  s_{0},...,s\right)  \in
H_{fnc}^{n}:\phi_{1}^{2}(s)\neq v_{N}=\widehat{\sigma}^{2}(s)\text{.}%
\]
Consequently \textit{every} $\overline{\sigma}^{2}$ and corresponding profile
$\overline{\sigma}$ is nonpositional.
\end{proof}

\subsection{Case $\left\vert V\right\vert \geq3$ and $\varepsilon
=0$\label{sec0403}}

\begin{proposition}
\label{prop0304}Let $G=(V,E)$ with $\left\vert V\right\vert \geq3$ and
$\varepsilon=0$. Then in every game $\Gamma_{N}(G|s_{0},\gamma,\varepsilon)$
there exists always a nonpositional trigger strategies profile $\overline
{\sigma}$, for all $\gamma\in(0,1)$.
\end{proposition}

\begin{proof}
Given $\left\vert V\right\vert \geq3$, from any starting $s_{0}$ it is
possible to reach at some $t$\ state $s=\left(  v_{1},v_{2},...,v_{N}%
,N\right)  \in S_{nc}\ $(so $R$ has the next move) such that $v_{1}\neq v_{2}$
and $v_{N}\in N\left(  v_{2}\right)  $. Now, under no CR-optimal
$\widehat{\sigma}^{N}$ $R$ ever moves to $v_{2}$. In game $\Gamma_{N}%
^{1}(G|s_{0},\gamma,\varepsilon)$ though $P_{-1}$ \emph{can}, under optimal
play move $R$ to $v_{2}$ since then he has a minimum loss of $\frac
{\varepsilon}{N-k-1}\gamma^{t}=0$, where $k$ is the number of $P_{-1}$'s cop
tokens located at $v_{2}$, including $C_{2}$. Thus,
\[
\exists\phi_{1}^{N}:\forall\widehat{\sigma}^{N}\in\widehat{\Sigma}^{N},\forall
s_{0}\in S_{nc},\exists h=\left(  s_{0},...,s\right)  \in H_{fnc}^{N}:\phi
_{1}^{N}(s)\neq\widehat{\sigma}^{N}(s)\text{;}%
\]
i.e., there always exists a nonpositional $\overline{\sigma}^{N}$ and a
corresponding nonpositional $\overline{\sigma}$.
\end{proof}

\subsubsection{Case:$\ c\left(  G\right)  \leq N-1$\label{sec040301}}

In this part of the paper we connect positionality of $\overline{\sigma}$ to
the cop number $c\left(  G\right)  $ of graph $G$. First we examine the case
where $G$ is a path (hence $c\left(  G\right)  =1$)\ with $\left\vert
V\right\vert \geq3$. \footnote{In a sense this proposition can be seen as an
extension of Proposition \ref{prop0303a}, part 2, to paths with $|V|>2$.}

\begin{proposition}
\label{Proppospathse=0}Let $G$ be a path with $\left\vert V\right\vert \geq3$;
let $\varepsilon=0$. Then $\Gamma_{N}(G|s_{0},\gamma,\varepsilon)$ has a
positional trigger strategies profile $\overline{\sigma}$ iff (i)
$s_{0}\mathbf{\ }$is\textbf{ }such that all cops are to one side of the
robber, and (ii) $\gamma\in\left(  0,\frac{1}{N-1}\right]  $.
\end{proposition}

\begin{proof}
Let $S_{nc}^{\prime}$ denote the set of states where the robber is between
some cops and $S_{nc}^{\prime\prime}=S_{nc}\backslash S_{nc}^{\prime}$ the set
of states where all cops are to one side of the robber. In Part I we show that
in every game $\Gamma_{N}(G|s_{0},\gamma,\varepsilon)$ with $s_{0}\in
S_{nc}^{\prime}$ every $\overline{\sigma}$ is nonpositional. In part II then
we show that, if $s_{0}\in S_{nc}^{\prime\prime}$, then a positional
$\overline{\sigma}$ exists iff $\gamma\in\left(  0,\frac{1}{N-1}\right]  $.
The combination of Parts I and II yields the result sought.

\bigskip

\noindent\underline{I. Initial states $s_{0}\in S_{nc}^{\prime}$}. From any
such $s_{0}$ we can always reach, at say time $t$, a state $\widetilde{s}$ in
which $R$ has the move, some cops are immediately to his left and the rest are
immediately to his right. Now, every CR-optimal robber strategy $\widehat
{\sigma}^{N}$ at state $\widetilde{s}$ dictates that the robber stays in
place. In game $\Gamma_{N}^{1}(G|s_{0},\gamma,\varepsilon)$ on the contrary,
every optimal strategy $\phi_{1}^{N}$ at $\widetilde{s}$ dictates that the
robber moves into the vertex not occupied by $C_{1}$, because this yields a
minimum loss of $\varepsilon\gamma^{t}=0$ for $P_{-n}$, whereas otherwise
$C_{1}$ optimally captures right after and $P_{-n}$'s loss is $(1-\varepsilon
)\gamma^{t+1}=\gamma^{t+1}>0$. Thus,
\[
\forall\phi_{1}^{N}\in\Phi_{1}^{N}\text{, }\forall\widehat{\sigma}^{N}%
\in\widehat{\Sigma}^{N}\text{, }\forall s_{0}\in S_{nc}^{\prime}\text{,
}\exists h=(s_{0},...,\widetilde{s})\in H_{fnc}^{N}:\phi_{1}^{N}%
(s)\neq\widehat{\sigma}^{N}(s)\text{.}%
\]
Hence in this case, every strategy $\overline{\sigma}^{N}$ and corresponding
profile $\overline{\sigma}$ is nonpositional.

\bigskip

\noindent\underline{II. Initial states $s_{0}\in S_{nc}^{\prime\prime}$}. Let
(without loss of generality) $S_{ncl}^{\prime\prime}\subset S_{nc}%
^{\prime\prime}$ be the states where all cops are to the left of $R$; for
every $s_{0}\in S_{ncl}^{\prime\prime}$ let $S_{ncl}^{\prime\prime}\left(
s_{0}\right)  $ be the set of states that can occur starting from $s_{0}$,
i.e.,
\[
S_{ncl}^{\prime\prime}\left(  s_{0}\right)  :=\left\{  s:\exists h=\left(
s_{0},...,s\right)  \in H_{fnc}\right\}  .
\]
Observe that, for all $s_{0}\in S_{ncl}^{\prime\prime}$, $S_{ncl}%
^{\prime\prime}\left(  s_{0}\right)  =S_{ncl}^{\prime\prime}$. Then existence
of a positional $\overline{\sigma}$ implies:%
\begin{equation}
\forall m,n\in\{1,...,N\} \text{, }\exists\phi_{m}^{n},\widehat{\sigma}%
^{n}:\forall s\in S_{ncl}^{\prime\prime}\cap S^{n}\text{, }\phi_{m}%
^{n}(s)=\widehat{\sigma}^{n}(s)\text{.} \label{eq011a}%
\end{equation}

Given $G$ is a path, for any $s\in S_{ncl}^{\prime\prime}$ and under every
CR-optimal profile $\widehat{\sigma}$, the cop $\widehat{C}\left(  s\right)  $
that captures is the one that is \textquotedblleft closer\textquotedblright%
\ to\emph{ }$R$, taking also into account whose turn is to move and at time
$\widehat{T}\left(  s\right)  $. Let $\widehat{\sigma}_{\ast}=(\widehat
{\sigma}_{\ast}^{1},...,\widehat{\sigma}_{\ast}^{N})$ be the CR-optimal
profile where, $C_{n}$ ($n\in\{1,2,...,N-1\}$) always moves towards $R$, and
$R$ moves away from the cops, reaches the\ end of the path and waits there
until capture. Instead of (\ref{eq011a}) we show the following which is
equivalent:%
\begin{equation}
\forall m,n\in\{1,...,N\} \text{, }\exists\phi_{m}^{n}:\forall s\in
S_{ncl}^{\prime\prime}\cap S^{n}\text{, }\phi_{m}^{n}(s)=\widehat{\sigma
}_{\ast}^{n}(s)\text{.} \label{eq012a}%
\end{equation}

Now fix an $m\in\{1,2,...,N-1\}$ for the rest of the proof and consider game
$\Gamma_{N}^{m}(G|s_{0},\gamma,\varepsilon)$. We partition $S_{ncl}%
^{\prime\prime}$ into two mutually disjoint sets $S_{A},S_{B}$ defined below
and examine each case separately.
\begin{align*}
S_{A}  &  :=\left\{  s\in S_{ncl}^{\prime\prime}:\text{under (every) }%
\widehat{\sigma}\text{, cop }C_{n}\text{ (}n\neq m\text{)\ captures (i.e.,
}\widehat{C}\left(  s\right)  =C_{n}\text{) at }\widehat{T}(s)\right\}  ,\\
S_{B}  &  :=\left\{  s\in S_{ncl}^{\prime\prime}:\text{under (every) }%
\widehat{\sigma}\text{, cop }C_{m}\text{ captures (i.e., }\widehat{C}\left(
s\right)  =C_{m}\text{) at }\widehat{T}(s)\right\}  \text{.}%
\end{align*}

\noindent\underline{\textbf{Case II.A: }$s\in S_{A}$}. For any $s\in S_{A}$,
if $P_{-m}$ uses the chosen CR-optimal (cop and robber) strategies
$\widehat{\sigma}_{\ast}^{n}$ ($n\in\{1,...,N\} \backslash m$) he can force a
$C_{n}$ capture at time $\widehat{T}(s)$ for any strategy of $P_{m}$ and get
his minimum loss of $\varepsilon\gamma^{\widehat{T}(s)}=0$. Given this
strategy of $P_{-m}$, $P_{m}$ cannot affect the outcome. Thus \emph{any}
strategy is optimal for him and so is the CR-optimal strategy $\widehat
{\sigma}_{\ast}^{m}$. Hence%
\begin{equation}
\forall n\in\{1,...,N\} \text{ }\exists\phi_{m}^{n}:\forall s\in S_{A}\cap
S^{n}\text{, }\phi_{m}^{n}(s)=\widehat{\sigma}_{\ast}^{n}(s)\text{.}
\label{eq1121}%
\end{equation}

\noindent\underline{\textbf{Case II.B: }$s\in S_{B}$}. Assume $P_{m}$ uses
$\widehat{\sigma}_{\ast}^{m}$ for $C_{m}$ and consider the options of $P_{-m}%
$. It can be seen that, depending on the state $s$, there exist only two
possibilities, which partition further $S_{B}$ as follows:

\begin{enumerate}
\item States $s\in S_{B_{1}}:P_{m}$ can force a \emph{pure} $C_{m}$ capture,
for \emph{any} strategy of $P_{-m}$, and

\item States $s\in S_{B_{2}}:P_{-m}$ can effect a \emph{joint capture}, i.e.,
one involving $C_{m}$ and some $P_{-m}$ cop tokens.
\end{enumerate}

If $s\in S_{B_{1}}$, then under optimal play (in $\Gamma_{N}^{m}%
(G|s_{0},\gamma,\varepsilon)$) $C_{m}$ chases $R$ till the end of the path and
captures him at time $\widehat{T}\left(  s\right)  $; this describes the
optimal strategies for $C_{m}$ and $R$. The remaining tokens $C_{n}$ cannot
affect the outcome. Thus, any strategy is optimal for them and so is the
chosen $\widehat{\sigma}_{\ast}^{n}$. Hence%
\begin{equation}
\forall n\in\{1,...,N\} \text{ }\exists\phi_{m}^{n}:\forall s\in S_{B_{1}}\cap
S^{n}\text{, }\phi_{m}^{n}(s)=\widehat{\sigma}_{\ast}^{n}(s)\text{.}
\label{eq1122}%
\end{equation}

Let now $s\in S_{B_{2}}$. First note that the only optimal strategies for
$P_{m}$ in this case are strategies $\widehat{\sigma}^{m}\in\widehat{\Sigma
}^{m}$. Indeed, and for any strategy $\sigma^{-m}$ of $P_{-m}$, if $P_{m}$
uses any $\widehat{\sigma}^{m}$ for $C_{m}$ the outcome is either a pure
$C_{m}$ capture, at the fastest possible time, or a joint capture, at the
fastest possible time. Any other strategy of $P_{m}$ leads to suboptimal for
him outcomes, even to pure $C_{n}$ ($n\neq m$) capture.

Assuming $P_{m}$ uses $\widehat{\sigma}_{\ast}^{m}$ for $C_{m}$, then $P_{-m}$
can effect a joint capture at some time $t$. First note that a joint capture
can only happen after a move by $R$ and only if he deviates from
$\widehat{\sigma}^{N}$, and second that this can only be at a time
$t<\widehat{T}\left(  s\right)  $.\footnote{The latter is a consequence of the
following: In a \emph{pure }$C_{m}$ capture under every $\widehat{\sigma}$,
only the moves of $C_{m}$ and $R$ are relevant (i.e., the remaining cops
cannot affect the outcome). Thus if $P_{-m}$ could effect a joint capture at
$t>\widehat{T}\left(  s\right)  $, given $C_{m}$ follows $\widehat{\sigma}%
^{m}$, he would be able to do so only due to moves of $R$. But if $R$ alone
could achieve capture later than $\widehat{T}\left(  s\right)  $, when $C_{m}$
uses $\widehat{\sigma}^{m}$, then $\widehat{T}\left(  s\right)  $ would not be
the optimal CR time, which is a contradiction.} Moreover it is clear that
$P_{-m}$ always prefers the joint capture which: (i)\ happens at the latest
possible time, call it $\widetilde{T}\left(  s\right)  $ and (ii)\ involves
the largest possible number of his cops, call it $\widetilde{K}\left(
s\right)  $. Now note that both these maximum values are achieved by following
$\widehat{\sigma}_{\ast}^{-m}$ until $\widetilde{T}\left(  s\right)  -1$, at
which time the robber is at the path end and $\widetilde{K}\left(  s\right)
+1$ cops are next to him\ \ and letting $R$ fall on the cops at $\widetilde
{T}\left(  s\right)  $; call this strategy $\widetilde{\sigma}^{-m}$. In this
case $P_{-m}$'s loss is
\[
\frac{1-\varepsilon}{\widetilde{K}\left(  s\right)  +1}\gamma^{\widetilde
{T}\left(  s\right)  }=\frac{1}{\widetilde{K}\left(  s\right)  +1}%
\gamma^{\widetilde{T}\left(  s\right)  }.
\]
Alternatively, $P_{-m}$ can choose to stick to $\widehat{\sigma}_{\ast}^{-m}$
until the end and let $C_{m}$ capture at $\widehat{T}\left(  s\right)  $.
Since at $\widehat{T}\left(  s\right)  $ (resp. at $\widetilde{T}\left(
s\right)  $)$\ C_{m}$ (resp. $R$)\ has the move, we have $\widehat{T}\left(
s\right)  =\widetilde{T}\left(  s\right)  +m$. In this case $P_{-m}$'s loss
is
\[
(1-\varepsilon)\gamma^{\widehat{T}\left(  s\right)  }=\gamma^{\widehat
{T}\left(  s\right)  }.
\]
Then $\widehat{\sigma}_{\ast}^{-m}$ is optimal for $P_{-m}$ iff
\begin{equation}
\gamma^{\widehat{T}\left(  s\right)  }\leq\frac{1}{\widetilde{K}\left(
s\right)  +1}\gamma^{\widetilde{T}\left(  s\right)  }\Leftrightarrow\gamma
^{m}\leq\frac{1}{\widetilde{K}\left(  s\right)  +1}\Leftrightarrow\gamma
\leq\left(  \frac{1}{\widetilde{K}\left(  s\right)  +1}\right)  ^{1/m}\text{.}
\label{eq1111}%
\end{equation}
For a positional $\overline{\sigma}$ to exist (\ref{eq1111})\ must hold for
all $s\in S_{B_{2}}$ and thus also for the minimum of $\left(  \frac
{1}{\widetilde{K}\left(  s\right)  +1}\right)  ^{1/m}$. This quantity is
increasing in $m$ and decreasing in $\widetilde{K}\left(  s\right)  $ and thus
takes its minimum for $m=1$ and $\widetilde{K}\left(  s\right)  =N-2$, i.e.,
when $C_{m}=C_{1}$ and at $\widetilde{T}\left(  s\right)  $ all cops are next
to the robber. Then (\ref{eq1111})\ becomes%
\begin{equation}
\gamma\leq\frac{1}{N-1}\text{.} \label{eq1112}%
\end{equation}
Hence, under \emph{and only under (\ref{eq1112}) }we have%
\begin{equation}
\forall n\in\{1,...,N\} \text{ }\exists\phi_{m}^{n}:\forall s\in S_{B_{2}}\cap
S^{n}\text{, }\phi_{m}^{n}(s)=\widehat{\sigma}_{\ast}^{n}(s)\text{.}
\label{eq1123}%
\end{equation}

Given $S_{B}=S_{B_{1}}\cup S_{B_{2}}$, $S_{ncl}^{\prime\prime}=S_{A}\cup
S_{B}$ and combining (\ref{eq1121}), (\ref{eq1122})\ and (\ref{eq1123}) yields
that a positional trigger strategies profile $\overline{\sigma}$ exists iff
(\emph{\ref{eq1112}}) holds.
\end{proof}

\begin{remark}
\label{Remneo01}\normalfont We now move on to graphs other than paths. It is
not hard to see that, for any such graph and pair of noncapture states
$s_{0},s$, there exists a history $h$ which starts at $s_{0}$ and ends at $s$.
I.e.,%
\[
\forall s_{0},s\in S_{nc}\text{, }\exists h=\left(  s_{0},...,s\right)  \in
H_{fnc}%
\]
This means that, \emph{for every }$s_{0}$, the set of endstates of \emph{all}
finite noncapture histories starting at $s_{0}$ is exactly $S_{nc}$. Thus
Conditions \textbf{A1}, \textbf{A2} for positional and nonpositional
respectively profiles reduce to:
\end{remark}

\begin{condition}
\label{condposprofNnew copy(1)}Let $\overline{\sigma}=(\overline{\sigma}%
^{1},...,\overline{\sigma}^{N})$ be a trigger strategies profile in
$\Gamma_{N}(G|s_{0},\gamma,\varepsilon)$, with $\overline{\sigma}^{n}$
consisting of $\left(  \phi_{m}^{n}\right)  _{m=1}^{N}$. Then $\overline
{\sigma}$ is positional (resp. nonpositional) iff \textbf{B1} (resp.
\textbf{B2}) holds:%
\begin{align}
\mathbf{B1}  &  :\forall n,m\in\{1,...,N\} \text{, }\exists\widehat{\sigma
}^{n}\in\widehat{\Sigma}^{n}:\forall s\in S^{n}\cap S_{nc}\text{, }\phi
_{m}^{n}(s)=\widehat{\sigma}^{n}(s),\label{eq011}\\
\mathbf{B2}  &  :\exists n,m\in\{1,...,N\}:\forall\widehat{\sigma}^{n}%
\in\widehat{\Sigma}^{n}\text{, }\exists s\in S^{n}\cap S_{nc}:\phi_{m}%
^{n}(s)\neq\widehat{\sigma}^{n}(s). \label{eq012}%
\end{align}

\end{condition}

The next proposition settles the issue for all rest graphs $G$ with $c\left(
G\right)  \leq N-1$.

\begin{proposition}
\label{prop0308}Consider $\Gamma_{N}(G|s_{0},\gamma,\varepsilon)$ where
$G=(V,E)$ is not a path, $c\left(  G\right)  \leq N-1$ and $\varepsilon=0$.
Then every trigger strategies profile $\overline{\sigma}$ is nonpositional for
all $\gamma\in(0,1)$.
\end{proposition}

\begin{proof}
Given $G$ is not a path, it is $|V|\geq3$ and, $G$ contains, or is equal to
either the clique $\mathcal{K}_{3}$ with $E=\left\{  \left\{  v_{1}%
,v_{2}\right\}  ,\left\{  v_{2},v_{3}\right\}  ,\left\{  v_{3},v_{1}\right\}
\right\}  $, or the star $\mathcal{S}_{3}$ $E=\left\{  \left\{  v_{1}%
,v_{2}\right\}  ,\left\{  v_{1},v_{3}\right\}  ,\left\{  v_{1},v_{4}\right\}
\right\}  $. By Remark \ref{Remneo01} then we have that, for any initial state
$s_{0}$, the state $\widetilde{s}$ where, some of the cops are on vertex
$v_{2}$, the rest are on $v_{3}$, the robber is on $v_{1}$ and it is the
robber's turn to move can always occur. Moreover, $c(G)\leq N-1$ means that,
starting from $\widetilde{s}$, capture occurs under CR-optimal play by
$\widehat{C}(\widetilde{s})$; let $\widehat{C}(\widetilde{s})=C_{m}$.

Given Proposition \ref{prop0304}, to show the current claim suffices to show
there exist no positional profiles $\overline{\sigma}$. Assume (towards
contradiction) there exists positional profile $\overline{\sigma}$. Now
consider game $\Gamma_{N}^{m}(G|s_{0},\gamma,\varepsilon)$ and the case where
state $\widetilde{s}$ has been reached. Then it must be that, there exist
optimal $\phi_{m}^{n}$ such that, Condition \textbf{B1} is satisfied for that
particular $m$ and $s=\widetilde{s}$. However, if a CR-optimal strategy
$\widehat{\sigma}^{n}$ is used for every $n\in\{1,...,N\}$, then starting from
$\widetilde{s}$, profile $\widehat{\sigma}$ leads to capture by $C_{m}$ at
$\widehat{T}(\widetilde{s})$, in which case $P_{-m}$'s loss will be
$\gamma^{\widehat{T}(\widetilde{s})}>0$. But we know that, starting from
$\widetilde{s}$, $P_{-m}$ can achieve his minimum loss of $0$ by moving $R$
into whichever of $v_{2}$ or $v_{3}$ does not contain $C_{m}$. Hence using
$\widehat{\sigma}^{n}$ for every $n\in\{1,...,N\}$ is suboptimal for $P_{-m}$
when starting from $\widetilde{s}$. Thus, there exists no positional
$\overline{\sigma}$.
\end{proof}

\subsubsection{Case:$\ c\left(  G\right)  >N-1$\label{sec040302}}

In this section we use the \emph{state cop number }$c_{N}\left(  G|s\right)  $
defined in Appendix \ref{chapA}. Thus we urge the reader to study this part
before proceeding.

For any given $N$, we define $\mathcal{G}(N)$ by
\begin{equation}
\mathcal{G}(N):=\{G:c(G)>N-1\}\text{,} \label{eq0H3}%
\end{equation}
i.e., the set of graphs examined in this section\footnote{It is a well known
result \cite{Aigner84} that, for every $k\in%
%TCIMACRO{\U{2115} }%
%BeginExpansion
\mathbb{N}
%EndExpansion
$, there exists a graph $G$ with $c(G)>k$.}\ or, equivalently by
(\ref{eq0H2}), the set
\begin{equation}
\mathcal{G}(N):=\{G:\exists s\text{ with }c(G|s)=\infty\}\text{.} \label{eqG}%
\end{equation}
Elements of $\mathcal{G}(3)$ include Dodecahedron and the Petersen graph (both
of which have $c(G)=3$) as well as any other graph resulting by
\emph{bridging}\ either of these to another graph.\footnote{Recall that a
\textquotedblleft bridge\textquotedblright\ is an edge whose deletion results
to a disconnected graph. Then our claim follows from the (easy to see) fact
that, if a graph $G$ with $c(G)=k$ is bridged to another graph, the resulting
graph $H$ will have $c(H)\geq k$.}\ The following propositions involve sets of
graphs that form a partition of $\mathcal{G}(N)$. Moreover, we sometimes
simply write $\mathcal{G}$ rather than $\mathcal{G}(N)$ and likewise for its constituents.

Given $N$, consider the following subsets of $\mathcal{G}$:

\begin{enumerate}
\item $\mathcal{G}_{1}$ consists of those graphs in $\mathcal{G}$ where there
exists a state starting from which, cooperation of two, or more cops, up to
$N-1$ is necessary and sufficient to ensure capture in $\breve{\Gamma}_{N}%
(G)$. I.e.,%
\begin{equation}
\mathcal{G}_{1}:=\{G\in\mathcal{G}:\exists s\text{ with }c(G|s)=k\in\left\{
2,...,N-1\right\}  \}. \label{eqG1}%
\end{equation}

\item $\mathcal{G}_{1}^{\prime}$ is the complement of $\mathcal{G}_{1}$ (with
respect to $\mathcal{G}$). Here, at every noncapture state, either the robber
evades under CR-optimal play, or there exists a single cop who can ensure
capture. I.e.,
\begin{equation}
\mathcal{G}_{1}^{\prime}:=\mathcal{G}\backslash\mathcal{G}_{1}=\{G\in
\mathcal{G}:\forall s\in S_{nc}\text{, }c(G|s)\in\left\{  1,\infty\right\}
\}. \label{eqG1'}%
\end{equation}

\item Finally, $\mathcal{G}_{2}$ consists of graphs such that, when it is the
robber's turn to move, he can always evade. I.e.,
\begin{equation}
\mathcal{G}_{2}:=\{G\in\mathcal{G}:\forall s\in S^{N}\cap S_{nc}\text{ it is
}c(G|s)=\infty\}. \label{eqG2a}%
\end{equation}
As we will shortly see, $\mathcal{G}_{2}$ is a subset of $\mathcal{G}%
_{1}^{\prime}$.
\end{enumerate}

Graphs in $\mathcal{G}_{1}(3)$ are typically graphs of $\mathcal{G}(3)$
containing as subgraphs cycles of length $l\geq4$. Some graphs in
$\mathcal{G}_{1}^{\prime}(3)$ are those resulting by bridging Dodecahedron or
Petersen with paths or trees. Some graphs in $\mathcal{G}_{2}(3)$ are
Dodecahedron and Petersen themselves.

The next proposition concerns $\mathcal{G}_{1}$. Note that holds, not only for
$\varepsilon=0$ but for every $\varepsilon\in\left[  0,\frac{1}{2}\right]
$.\footnote{However, for $\varepsilon>0$ the claim has been already shown in
Proposition \ref{prop0303}.}

\begin{proposition}
\label{PropTza}Consider $\Gamma_{N}(G|s_{0},\gamma,\varepsilon)$ with
$G\in\mathcal{G}_{1}$ and $\varepsilon\in\left[  0,\frac{1}{2}\right]  $. Then
every profile $\overline{\sigma}$ is nonpositional for all $\gamma\in(0,1)$.
\end{proposition}

\begin{proof}
Let $s\in S_{nc}$ with $c(G|s)=k\in\left\{  2,...,N-1\right\}  $; let
$\widehat{C}(s)=C_{m}$. Consider game $\Gamma_{N}^{m}(G|s,\gamma,\varepsilon
)$. Given $c(G|s)\geq2$, $P_{-m}$ can enforce robber evasion in $\Gamma
_{N}^{m}(G|s,\gamma,\varepsilon)$. Hence it cannot be $\phi_{m}^{n}%
(s_{t})=\widehat{\sigma}^{n}(s_{t})$ for all $n\in\{1,...,N-1\}\backslash m$
and all states $s_{t}\in S^{n}$ following $s$, because this allows $P_{m}$ to
capture$\ R$ and is suboptimal play for $P_{-m}$. Hence, there exists at least
one $n\in\{1,...,N-1\}\backslash m$ and at least one $s_{t}\in S^{n}$
following $s$ such that, $\phi_{m}^{n}(s_{t})\neq\widehat{\sigma}^{n}(s_{t})$,
$\forall\widehat{\sigma}^{n}\in\widehat{\Sigma}^{n}$, leading to the result sought.
\end{proof}

In the following lemma we show that $\mathcal{G}_{1}\cap\mathcal{G}%
_{2}=\emptyset$ and hence $\mathcal{G}_{2}\subset\mathcal{G}_{1}^{\prime}$.

\begin{lemma}
$\mathcal{G}_{1}(N)\cap\mathcal{G}_{2}(N)=\emptyset$.
\end{lemma}

\begin{proof}
Given $N$, assume on the contrary $\mathcal{G}_{1}\cap\mathcal{G}_{2}%
\neq\emptyset$ and let $G\in\mathcal{G}_{1}\cap\mathcal{G}_{2}$.
$G\in\mathcal{G}_{1}$ means there exists state $s$ such that $c(G|s)=k\in
\left\{  2,...,N-1\right\}  $.$\ $This in its turn means that, at $s$, there
exists no cop who (a) is located next to $R$ and (b) plays before $R$ (because
otherwise it would have been $c(G|s)=1$). This again means that, starting from
$s$ and irrespective of the moves of the cops playing before $R$, $R$ will
have the chance to move i.e., the game will reach a state $s^{\prime}\in
S^{N}\cap S_{nc}$. Now by assumption $G$ belongs also to $\mathcal{G}_{2}$ and
thus $c(G|s^{\prime})=\infty$ (i.e., starting from $s^{\prime}$ $R$ evades
under CR-optimal robber play). But then (from the previous argument) we have
that, under CR-optimal robber play $R$ evades also from $s$ and thus
$c(G|s)=\infty$, which however contradicts the assumption $c(G|s)=k\in\left\{
2,...,N-1\right\}  $. Hence there exists no such graph $G$ and thus
$\mathcal{G}_{1}\cap\mathcal{G}_{2}=\emptyset$.
\end{proof}

Furthermore $\mathcal{G}_{1}^{\prime}$ clearly contains graphs that do not
belong in $\mathcal{G}_{2}$\ and thus we have the following.

\begin{corollary}
$\mathcal{G}_{2}(N)\subset\mathcal{G}_{1}^{\prime}(N)$.
\end{corollary}

\begin{remark}
\label{RemF}\normalfont An important fact is the following: $c(G|s)=\infty$
$\ $for all $s\in S^{N}\cap S_{nc}$ implies that, under CR-optimal robber
play, the \emph{only} states that can lead to capture in any $G\in
\mathcal{G}_{2}$ are those where \emph{a cop is next to }$R$\emph{ and moves
before him}.\footnote{Note that this type of states are a trivial case of
states $s$ with $c(G|s)=1$ existing in every graph.}
\end{remark}

Next we identify one more set of graphs and respective games where positional
profiles $\overline{\sigma}$ exist.

\begin{proposition}
\label{Existposc(G)>N-1}Consider $\Gamma_{N}(G|s_{0},\gamma,\varepsilon)$ with
$G\in\mathcal{G}_{2}$ and $\varepsilon=0$. Then there exists a positional
profile $\overline{\sigma}$ for all $\gamma\in(0,1)$.
\end{proposition}

\begin{proof}
Fix an $m\in\{1,...,N-1\}$ for the rest of the proof and consider $\Gamma
_{N}^{m}(G|s_{0},\gamma,\varepsilon)$. We partition $S_{nc}$ into three
mutually disjoint sets $S_{A},S_{B}$ and $S_{C}$ defined below and examine
each case separately.
\begin{align*}
S_{A}  &  :=\left\{  s\in S_{nc}:\text{under (every) }\widehat{\sigma}\text{
the robber evades}\right\}  \text{,}\\
S_{B}  &  :=\left\{  s\in S_{nc}:\text{under (every) }\widehat{\sigma}\text{,
cop }C_{n}\text{ (}n\neq m\text{)\ captures (i.e., }\widehat{C}\left(
s\right)  =C_{n}\text{) at }\widehat{T}(s)\right\}  ,\\
S_{C}  &  :=\left\{  s\in S_{nc}:\text{under (every) }\widehat{\sigma}\text{,
cop }C_{m}\text{ captures (i.e., }\widehat{C}\left(  s\right)  =C_{m}\text{)
at }\widehat{T}(s)\right\}  \text{.}%
\end{align*}

\noindent\underline{\textbf{Case A: }$s\in S_{A}$}. For any $s\in S_{A}$, if
$P_{-m}$ uses any CR-optimal cop and robber strategies $\widehat{\sigma}^{n}$
($n\in\{1,...,N\} \backslash m$) he can force evasion of $R$ for any strategy
of $P_{m}$ and get his minimum loss of $0$. Given $P_{-m}$'s strategy, $P_{m}$
cannot affect the outcome. Thus \emph{any} strategy is optimal for him and so
is any CR-optimal strategy $\widehat{\sigma}^{m}$. Hence%
\begin{equation}
\forall n\in\{1,...,N\} \text{ }\exists\phi_{m}^{n},\widehat{\sigma}%
^{n}:\forall s\in S_{A}\cap S^{n}\text{, }\phi_{m}^{n}(s)=\widehat{\sigma}%
^{n}(s)\text{.} \label{eq012b}%
\end{equation}

\noindent\underline{\textbf{Case B: }$s\in S_{B}$}. By Remark \ref{RemF},
these are states such that: (i) a cop $C_{n}$ ($n\neq m$) is next the robber
and moves before him and (ii) if there exists another cop $C_{k}$ who is also
next the robber, then $C_{k}$ moves after $C_{n}$. Now, for any $s\in S_{B}$,
if $P_{-m}$ uses any CR-optimal strategies $\widehat{\sigma}^{n}$
($n\in\{1,...,N\}\backslash m$) $C_{n}$ captures at time $\widehat{T}(s)$ for
any strategy of $P_{m}$ and $P_{-m}$ gets his minimum loss of $\varepsilon
\gamma^{\widehat{T}(s)}=0$. Given $P_{-m}$'s strategy, $P_{m}$ cannot affect
the outcome. Thus any strategy is optimal for him and so is any CR-optimal
strategy $\widehat{\sigma}^{m}$. Hence%
\begin{equation}
\forall n\in\{1,...,N\}\text{ }\exists\phi_{m}^{n},\widehat{\sigma}%
^{n}:\forall s\in S_{B}\cap S^{n}\text{, }\phi_{m}^{n}(s)=\widehat{\sigma}%
^{n}(s)\text{.} \label{eq013}%
\end{equation}

\noindent\underline{\textbf{Case C: }$s\in S_{C}$}. By Remark \ref{RemF},
these are states such that: (i) cop $C_{m}$ is next to the robber and moves
before him and (ii) if there exists another cop $C_{n}$ who is also next the
robber, then $C_{n}$ moves after $C_{m}$. Now, for any $s\in S_{C}$, if
$P_{m}$ uses any CR-optimal cop strategy $\widehat{\sigma}^{m}$, then $C_{m}$
captures at time $\widehat{T}(s)$ for any strategy of $P_{-m}$ and $P_{m}$
gets his maximum gain of $(1-\varepsilon)\gamma^{\widehat{T}(s)}%
=\gamma^{\widehat{T}(s)}$. Given $P_{m}$'s strategy, $P_{-m}$ cannot affect
the outcome. Thus any strategy is optimal for him and so is any CR-optimal
strategies $\widehat{\sigma}^{n}$ ($n\in\{1,...,N\}\backslash m$). Hence%
\begin{equation}
\forall n\in\{1,...,N\}\text{ }\exists\phi_{m}^{n},\widehat{\sigma}%
^{n}:\forall s\in S_{C}\cap S^{n}\text{, }\phi_{m}^{n}(s)=\widehat{\sigma}%
^{n}(s)\text{.} \label{eq014}%
\end{equation}

Combining(\ref{eq012b})-(\ref{eq014}) and $S_{A}\cup S_{B}\cup S_{C}=S_{nc}$
we get the result sought.
\end{proof}

\medskip

Given $N$, let $\mathcal{G}_{2}^{\prime}$ be the complementary of
$\mathcal{G}_{2}$ in $\mathcal{G}_{1}^{\prime}$, i.e.,%
\begin{equation}
\mathcal{G}_{2}^{\prime}:=\mathcal{G}_{1}^{\prime}\backslash\mathcal{G}%
_{2}=\{G\in\mathcal{G}_{1}^{\prime}:\exists s\in S^{N}\text{ with }c(G|s)=1\}.
\label{eqG2'}%
\end{equation}

Summarizing to this point we have: (a) partitions $\mathcal{G=G}_{1}%
\cup\mathcal{G}_{1}^{\prime}$ and $\mathcal{G}_{1}^{\prime}=\mathcal{G}%
_{2}\cup\mathcal{G}_{2}^{\prime}$ and (b) $\mathcal{G}_{2}^{\prime}$ is the
last class of graphs remaining to examine.

In search of positional profiles $\overline{\sigma}$ within $\mathcal{G}%
_{2}^{\prime}$, a process of trial and error (coupled with some intuition) led
us to the following (and the last) partition of $\mathcal{G}_{2}^{\prime}$.
$\mathcal{G}_{3}$ is the subset of $\mathcal{G}_{2}^{\prime}$ with graphs
satisfying the property that:\ at \emph{every} state $s$ with $c(G|s)=1$, cop
$\widehat{C}_{m}(s)$ ($m\in\{1,...,N-1\}$) can always \emph{effect} capture
using $\widehat{\sigma}^{m}$, \emph{no matter what the rest players do};
$\mathcal{G}_{3}^{\prime}$ is the complement of $\mathcal{G}_{3}$ with respect
to $\mathcal{G}_{2}^{\prime}$. Formally, given $N$ and $\widehat{C}_{m}(s)$
($m\in\{1,...,N-1\}$) for all $s\in S_{nc}$, define:%
\begin{align}
\mathcal{G}_{3}  &  :=\left\{  G\in\mathcal{G}_{2}^{\prime}:\forall\left(
s\text{ with }c(G|s)=1\text{ and }\sigma^{-m}\right)  \text{ we have:
}(s,\widehat{\sigma}^{m},\sigma^{-m})\rightarrow\widehat{C}_{m}(s)\text{
capture}\right\}  ,\label{eqG3}\\
\mathcal{G}_{3}^{\prime}  &  :=\left\{  G\in\mathcal{G}_{2}^{\prime}%
:\exists\left(  s\text{ with }c(G|s)=1\text{ and }\sigma^{-m}\right)  \text{
such that: }(s,\widehat{\sigma}^{m},\sigma^{-m})\not \rightarrow \widehat
{C}_{m}(s)\text{ capture}\right\}  ; \label{eqG3prime}%
\end{align}
where $\rightarrow$ (resp. $\not \rightarrow $) means "leads to" (resp. does
not lead to).

The next proposition concerns graphs belonging to $\mathcal{G}_{3}^{\prime}$.
\footnote{One graph in $\mathcal{G}_{3}^{\prime}$ is a graph where the
Petersen graph is bridged to a path. Then a state satisfying the condition in
(\ref{eqG3prime}) is the one where $R$ is \textquotedblleft
crammed\textquotedblright\ between $C_{1}$ and $C_{2}$ on the path and it is
$R$'s turn to move. Under CR-optimal play $R$ stays put in the first move and
$C_{1}$ captures right after. However, $R$ can run into $C_{2}$ straight
ahead.}

\begin{proposition}
\label{Propnea01}Consider $\Gamma_{N}(G|s_{0},\gamma,\varepsilon)$ with
$G\in\mathcal{G}_{3}^{\prime}$ and $\varepsilon=0$. Then every profile
$\overline{\sigma}$ is nonpositional for all $\gamma\in(0,1)$.
\end{proposition}

\begin{proof}
Let $G\in\mathcal{G}_{3}^{\prime}$, $s$ satisfying the conditions in
(\ref{eqG3prime}) and $\widehat{C}_{m}(s)$, $\widehat{T}(s)$ as known.
Consider game $\Gamma_{N}^{m}(G|s,\gamma,\varepsilon)$. Starting from $s$,
under every profile $\widehat{\sigma}$, $\widehat{C}_{m}(s)$ captures at
$\widehat{T}(s)$ and $P_{-m}$'s loss is $(1-\varepsilon)\gamma^{\widehat
{T}(s)}=\gamma^{\widehat{T}(s)}$. Now, if $\sigma^{-m}$ is such that
$(s,\widehat{\sigma}^{m},\sigma^{-m})$ does not lead to $\widehat{C}_{m}(s)$
capture, then, using $\sigma^{-m}$ $P_{-m}$ can force, either evasion of $R$,
or capture by $C_{n}$ with $n\neq m$, achieving in both cases a minimum loss
of $0<\gamma^{\widehat{T}(s)}$. Hence, there exists at least one
$n\in\{1,...,N\} \backslash m$ and state $s_{t}\in S^{n}$ following $s$ where
$\phi_{m}^{n}(s_{t})\neq\widehat{\sigma}^{n}(s_{t})$, $\forall\widehat{\sigma
}^{n}\in\widehat{\Sigma}^{n}$ and thus, every profile $\overline{\sigma}$ is
nonpositional (and for all $\gamma\in(0,1)$).
\end{proof}

It remains to examine graphs in $\mathcal{G}_{3}$. As a first step, we aim to
elucidate the form of such graphs. The following lemma makes clearer the
distinction between sets $\mathcal{G}_{3}$ and $\mathcal{G}_{3}^{\prime}$.

\begin{lemma}
The following hold for any $N$:
\begin{align}
\text{If }G  &  \in\mathcal{G}_{3}\text{ then:\ }\forall\left(  s\in
S^{N}\text{ with }c(G|s)=1\text{ and }\sigma^{-m}\right)  \text{ we have:
}(s,\widehat{\sigma}^{m},\sigma^{-m})\rightarrow\widehat{C}_{m}(s)\text{
capture.}\label{eqG3new}\\
\text{If }G  &  \in\mathcal{G}_{3}^{\prime}\text{ then:\ }\exists\left(  s\in
S^{N}\text{ with }c(G|s)=1\text{ and }\sigma^{-m}\right)  \text{ such
that:\ }(s,\widehat{\sigma}^{m},\sigma^{-m})\not \rightarrow \widehat{C}%
_{m}(s)\text{ capture.} \label{eqG3primenew}%
\end{align}

\end{lemma}

\begin{proof}
Clearly (\ref{eqG3}) implies (\ref{eqG3new}). For (\ref{eqG3primenew}), let
$s,\sigma^{-m}$ satisfy the conditions in (\ref{eqG3prime}). If $s\in S^{N}$
then (\ref{eqG3primenew}) holds. If $s\notin S^{N}$, then starting from $s$
(and for any strategy of the cops)\ the game will certainly reach a state
$s^{\prime}\in S^{N}$ (i.e., the robber will move at least once) by the
following argument.

If the robber does not move at least once, then he is captured before he can
move. This means that: at $s$ a cop (a) is next to $R$ and (b)\ moves before
$R$. But then by definition this cop is $\widehat{C}_{m}(s)$, which
contradicts condition (\ref{eqG3prime}).

So, let $s^{\prime}$ be the \emph{first} state in $S^{N}$ that occurs starting
from $s$, under profile $(\widehat{\sigma}^{m},\sigma^{-m})$, where
$\sigma^{-m}$ is the same as in (\ref{eqG3prime}). Then $s^{\prime}$ clearly
satisfies (\ref{eqG3primenew}) for this same $\sigma^{-m}$.
\end{proof}

\begin{remark}
\label{RemG3}\normalfont Thus we see that for every $G\in\mathcal{G}_{3}$,
each state $s\in S^{N}\cap S_{nc}$ is such that, either (i) $c(G|s)=\infty$,
or (ii) $c(G|s)=1$ and $s$ satisfies condition (\ref{eqG3new}).
\end{remark}

Next we show that, for every state $s$ satisfying (\ref{eqG3new}), there is a
state $\widetilde{s}$ which satisfies the exact same conditions as $s$ and
differs from $s$ only in that the capturing cop $C_{m}$ is next to the robber.

\begin{lemma}
\label{Lemdipla}Let $s=(x^{1},...,x^{m},...,x^{N-1},u,N)\in S^{N}$ satisfying
(\ref{eqG3new}) for some $\widehat{C}_{m}(s)$ and $x^{m}\notin N(u)$. Then
there exists $\widetilde{s}=(x^{1},...,\widetilde{x}^{m},...,x^{N-1},u,N)\in
S^{N}$ such that (i)\ $\widetilde{x}^{m}\in N(u)$,\ (ii) $c(G|\widetilde
{s})=1$, (iii) $\widehat{C}(\widetilde{s})=C_{m}$ and (iv) \ for every
$\widetilde{\sigma}^{-m}$, $(\widetilde{s},\widehat{\sigma}^{m},\widetilde
{\sigma}^{-m})$ leads to $\widehat{C}_{m}(\widetilde{s})$ capture.
\end{lemma}

\begin{proof}
Let $\acute{\sigma}^{-m}$ be the profile where all players besides $C_{m}$
stay put. Then, starting from $s$ and under $(\widehat{\sigma}^{m}%
,\acute{\sigma}^{-m})$, there will come a time where $C_{m}$ reaches $R$'s
neighborhood (because otherwise $R$ could stay put indefinitely and evade) and
it is $R$'s turn to move. Let the respective state be $\widetilde{s}$. First
note that $\widetilde{s}$ has the required form; and $\widetilde{x}^{m}\in
N(u)$ i.e., condition (i) holds. Now, given $(s,\widehat{\sigma}^{m}%
,\sigma^{-m})$ leads to $\widehat{C}_{m}(s)$ capture for \emph{every}
$\sigma^{-m}$, this must be also true for \emph{any} profile $\sigma^{-m}$
that copies $\acute{\sigma}^{-m}$ until $\widetilde{s}$ occurs and follows
\emph{any} profile $\widetilde{\sigma}^{-m}$ thereafter. Thus, $(\widetilde
{s},\widehat{\sigma}^{m},\widetilde{\sigma}^{-m})$ leads to $\widehat{C}%
_{m}(s)$ capture for \emph{every} $\widetilde{\sigma}^{-m}$ and condition (iv)
also holds. Conditions (ii)-(iii) follow immediately from (iv).
\end{proof}

The following lemma reveals an important fact for graphs (as those in
$\mathcal{G}_{3}$) satisfying (\ref{eqG3new}).

\begin{lemma}
\label{LemG3}Let $G\in\mathcal{G}_{3}$. If $s=(x^{1},...,x^{m},...,x^{N-1}%
,u,N)\in S^{N}$ satisfies (\ref{eqG3new}), then:

\begin{enumerate}
\item $\left\vert N(u)\right\vert =1$ (i.e., $u$ is a leaf) and

\item if $\{v\}=N(u)$, then for all $s^{\prime}=(y^{1},...,y^{N-1},v,N)\in
S^{N}$ it is $c(G|s^{\prime})=\infty.$
\end{enumerate}
\end{lemma}

\begin{proof}
\textbf{1.} Assume towards contradiction $\left\vert N(u)\right\vert \geq2$.
We distinguish the following cases.

\begin{description}
\item[1.A.] $s$ is such that $x^{m}\in N(u)$. Then $N(u)$ contains at least
one more vertex $z$. There are two possibilities.

\item[(i)] $z$ is occupied by some cop $C_{i}\neq\widehat{C}_{m}(s)$. But this
violates the requirement that $(s,\widehat{\sigma}^{m},\sigma^{-m})$ leads to
$\widehat{C}_{m}(s)$ capture, for every $\sigma^{-m}$ since $R$ can run into
$C_{i}$ on his first move.

\item[(ii)] $z$ is not occupied by any cop. Moving some cop $C_{i}\neq
\widehat{C}_{m}(s)$ to vertex $z$ we create state $s^{\prime\prime}\in S^{N}$
which differs from $s$ only in $C_{i}$'s new location. Now we have (a)
$c(G|s^{\prime\prime})=1$ and (b) letting $\widehat{C}(s^{\prime\prime}%
)=C_{n}\ $(where $C_{n}$ may be $C_{m}$ or $C_{i}$) then (given $G\in
\mathcal{G}_{3}$) we must have that $(s^{\prime\prime},\widehat{\sigma}%
^{n},\sigma^{-n})$ leads to $C_{n}$ capture for every $\sigma^{-n}$. But (b)
is violated since $R$ can run into whichever cop is not $C_{n}$ (either
$C_{m}$ or $C_{i}$).

\item Hence, if $s$ is such that $x^{m}\in N(u)$ then $|N(u)|=1$ and $u$ is a leaf.

\item[1.B.] $s$ is such that $x^{m}\notin N(u)$. Then there exists some
$\widetilde{s}$ of the type described in Lemma \ref{Lemdipla} and we return to
Case A.

\item We conclude that if $s$ satisfies (\ref{eqG3new}) then $u$ is always a leaf.
\end{description}

\noindent\textbf{2.} Assume on the contrary that $v$, the neighbor of $u,$ is
such that there exists $s^{\prime}=(y^{1},...,y^{N-1},v,N)\in S^{N}$ with
$c(G|s^{\prime})\neq\infty$; then, since $G\in\mathcal{G}_{3}$, it must be
$c(G|s^{\prime})=1$; therefore, $s^{\prime}$ also satisfies (\ref{eqG3new}).
From Part 1 then we have that $v$ is a leaf; given $u$ is also a leaf, it
follows that $G$ is the path $\mathcal{P}_{2}$; but this contradicts
$G\in\mathcal{G}_{3}$.
\end{proof}

\begin{corollary}
Let $G\in\mathcal{G}_{3}$. Then $u$ is not a leaf iff%
\begin{equation}
\forall s=(x^{1},...,x^{N-1},u,N)\in S^{N}\text{ it is }c(G|s)=\infty.
\label{eqCor01}%
\end{equation}

\end{corollary}

\begin{proof}
The left to right implication follows from Lemma \ref{LemG3}; the reverse from
the fact that if $u$ is a leaf, then there always exist states $s$ with
$c(G|s)=1$ e.g those where a cop occupies $u$'s neighbor.
\end{proof}

The last two results lead directly to the following characterization of the
form of graphs in $\mathcal{G}_{3}$.

\begin{corollary}
\label{Corlast}Let $G\in\mathcal{G}_{3}$ and $u\in V(G)$. Then (i) either $u$
is a leaf, in which case its neighbor satisfies the condition in Part 2 of
Lemma \ref{LemG3} (ii) or $u$ is not a leaf and then condition (\ref{eqCor01}) holds.
\end{corollary}

\begin{remark}
\normalfont In other words, Corollary \ref{Corlast} means that, $\mathcal{G}%
_{3}$ consists of the graphs obtained by attaching to each graph
$G\in\mathcal{G}_{2}$ an arbitrary (but positive) number of leaves. Some
graphs in $\mathcal{G}_{3}$ then are graphs like Petersen, or Dodecahedron to
which some leaves have been attached to.
\end{remark}

The following proposition concludes this study.

\begin{proposition}
Consider $\Gamma_{N}(G|s_{0},\gamma,\varepsilon)$ with $G\in\mathcal{G}_{3}$
and $\varepsilon=0$. Then there exists a positional profile $\overline{\sigma
}$ iff $\gamma\in(0,\frac{1}{N-1}]$.
\end{proposition}

\begin{proof}
Take first any vertex $u$ which is not a leaf. Then condition (\ref{eqCor01})
holds. But (\ref{eqCor01}) is the same as the defining condition of
$\mathcal{G}_{2}$. Hence, the same analysis as that in Proposition
\ref{Existposc(G)>N-1} leads to the conclusion that: whenever the robber
occupies a vertex that is not a leaf, every $\widehat{\sigma}^{n}\in
\widehat{\Sigma}^{n}$ is optimal for each token $n\in\{1,...,N\}$ and in every
game $\Gamma_{N}^{m}(G|s_{0},\gamma,\varepsilon)$ for $m\in\{1,...,N\}$ and
$\gamma\in(0,1)$.

Thus, we only need to examine cases where the robber occupies an arbitrary
leaf of the graph. Fix an $m\in\{1,...,N-1\}$ for the rest of the proof and
consider $\Gamma_{N}^{m}(G|s_{0},\gamma,\varepsilon)$. For the arbitrary leaf
$u$ of $G$, let $S_{ncu}$ denote the set of all states $s\in S_{nc}$ where the
robber occupies $u$. We now partition $S_{ncu}$ as follows and consider for
each set of states players' $P_{-m}$ and $P_{m}$ optimal strategies.
\begin{align*}
S_{A}  &  :=\{s\in S_{ncu}:\text{under (every) }\widehat{\sigma}\text{ the
robber evades}\}\\
S_{B}  &  :=\left\{  s\in S_{ncu}:\text{under (every) }\widehat{\sigma}\text{,
cop }C_{m}\text{ captures (i.e., }\widehat{C}\left(  s\right)  =C_{m}\text{)
at }\widehat{T}(s)\right\} \\
S_{C}  &  :=\left\{  s\in S_{ncu}:\text{under (every) }\widehat{\sigma}\text{,
cop }C_{n}\text{ (}n\neq m\text{)\ captures (i.e., }\widehat{C}\left(
s\right)  =C_{n}\text{) at }\widehat{T}(s)\right\}  \text{.}%
\end{align*}

\noindent\underline{\textbf{Case A: }$s\in S_{A}$}. For any $s\in S_{A}$, if
$P_{-m}$ uses any CR-optimal (cop and robber) strategies $\widehat{\sigma}%
^{n}$ ($n\in\{1,...,N\} \backslash m$) he can force evasion of $R$ for any
strategy of $P_{m}$ and get his minimum loss of $0$. Given $P_{-m}$'s
strategy, $P_{m}$ cannot affect the outcome. Thus \emph{any} strategy is
optimal for him and so is any CR-optimal strategy $\widehat{\sigma}^{m}$.
Hence%
\begin{equation}
\forall n\in\{1,...,N\} \text{ }\exists\phi_{m}^{n},\widehat{\sigma}%
^{n}:\forall s\in S_{A}\cap S^{n}\text{, }\phi_{m}^{n}(s)=\widehat{\sigma}%
^{n}(s)\text{.} \label{eq025a}%
\end{equation}

\medskip

Let us pause here and consider the form of states $s=(x^{1},...,x^{N-1},u,n)$
\ (i.e., $R$ occupies $u$) where capture can occur under CR-optimal play. Let
$v$ denote $u$'s (unique) neighbor in $G$. Then it is not hard to see that,
under CR-optimal play, a capture can occur iff there exists at least one cop
$C_{i}$, who is (i) either already at $v$, i.e., $x^{i}=v$, or (ii) he can
cover $v$ in case $R$ moves there, i.e., either (ii.a) $d(x^{i},v)=1$, or
(ii.b) $d(x^{i},v)=2$ and $C_{i}$ moves before $R$. In any other case $R$ can
evade using $\widehat{\sigma}^{N}$. Furthermore note that, in all such states
$s$ it is $c(G|s)=1$ and $\widehat{C}_{m}(s)$ is the cop that is
\textquotedblleft closer\textquotedblright\ to $R$, taking into account also
whose turn is to move.\footnote{That is, (i) either $\widehat{C}_{m}(s)$ is
already at $v$ and for any other cop $C_{j}$ that might be also at $v$,
$\widehat{C}_{m}(s)$ moves before $C_{j}$, or (ii) $\widehat{C}_{m}(s)$ is at
distance $1$ from $v$ and for any other cop $C_{j}$ that might also be at
distance $1$ from $v$, $\widehat{C}_{m}(s)$ moves before $C_{j}$, or (iii)
$\widehat{C}_{m}(s)$ is at distance $2$ from $v$, no other cop's distance from
$v$ is less than $2$, and for any other cop $C_{j}$ that might also be at
distance $2$ from $v$, $\widehat{C}_{m}(s)$ moves before $C_{j}$ and
$\widehat{C}_{m}(s)$ moves before $R$.}

\medskip

\noindent\underline{\textbf{Case B: }$s\in S_{B}$}. Assume $P_{m}$ uses
$\widehat{\sigma}^{m}$ for $C_{m}$ and consider the options of $P_{-m}$. It
can be seen that, depending on the state $s$, there exist only two
possibilities, which partition further $S_{B}$ as follows:

\begin{enumerate}
\item States $s\in S_{B_{1}}:C_{m}$ effects a pure capture, for any strategy
of $P_{-m}$, and

\item States $s\in S_{B_{2}}:P_{-m}$ can effect a \emph{joint capture}, i.e.,
one involving $C_{m}$ and some $P_{-m}$ cop tokens.\footnote{Note that, the
expression $(s,\widehat{\sigma}^{m},\sigma^{-m})$ leads to $\widehat{C}%
_{m}(s)$ capture in (\ref{eqG3}) does not exclude the possibility of a joint
capture, which however involves $\widehat{C}_{m}(s)$ as well.}
\end{enumerate}

\noindent If $s\in S_{B_{1}}$, then under optimal play (in $\Gamma_{N}%
^{m}(G|s_{0},\gamma,\varepsilon)$) $C_{m}$ goes straight towards $R$, while
$R$ stays put and $C_{m}$ captures at time $\widehat{T}\left(  s\right)  $.
This describes the optimal strategies for $C_{m}$ and $R$. The remaining
tokens $C_{n}$ cannot affect the outcome. Thus, any strategy is optimal for
them and so is any $\widehat{\sigma}^{n}$. Hence%
\begin{equation}
\forall n\in\{1,...,N\} \text{ }\exists\phi_{m}^{n},\widehat{\sigma}%
^{n}:\forall s\in S_{B_{1}}\cap S^{n}\text{, }\phi_{m}^{n}(s)=\widehat{\sigma
}^{n}(s)\text{.} \label{eq028}%
\end{equation}
If $s\in S_{B_{2}}$, then the exact same analysis as in case II.B. part 2 of
Proposition \ref{Proppospathse=0} leads to the conclusion that, whereas for
$P_{m}$ using $\widehat{\sigma}^{m}$ is always optimal in $\Gamma_{N}%
^{m}(G|s_{0},\gamma,\varepsilon)$, $P_{-m}$, uses optimally $\widehat{\sigma
}^{-m}$ in $\Gamma_{N}^{m}(G|s_{0},\gamma,\varepsilon)$ for all $m\in
\{1,...,N-1\}$ and states $s$ iff%
\begin{equation}
\gamma\leq\frac{1}{N-1}\text{.} \label{eq031}%
\end{equation}
Thus we have that iff (\ref{eq031}) holds, then%
\begin{equation}
\forall n\in\{1,...,N\} \text{ }\exists\phi_{m}^{n},\widehat{\sigma}%
^{n}:\forall s\in S_{B_{2}}\cap S^{n}\text{, }\phi_{m}^{n}(s)=\widehat{\sigma
}^{n}(s)\text{.} \label{eq032}%
\end{equation}

\noindent\underline{\textbf{Case C: }$s\in S_{C}$}. In this case $P_{-m}$
employing CR-optimal strategies $\widehat{\sigma}^{n}$ for his cop and robber
tokens results to capture by $C_{n}$ and his best outcome, i.e., a loss of
$0$. Similarly $P_{m}$ loses anyhow so he may as well employ $\widehat{\sigma
}^{m}$ for his cop token $C_{m}$. Hence,%
\begin{equation}
\forall n\in\{1,...,N\} \text{ }\exists\phi_{m}^{n},\widehat{\sigma}%
^{n}:\forall s\in S_{C}\cap S^{n}\text{, }\phi_{m}^{n}(s)=\widehat{\sigma}%
^{n}(s)\text{.} \label{eq033}%
\end{equation}

\bigskip

Given $S_{A}\cup S_{B}\cup S_{C}=S_{ncv}$ for every leaf $v$ of $G$ and
relations (\ref{eq025a}), (\ref{eq028}), (\ref{eq032}) and (\ref{eq033}), we
have that, iff (\ref{eq031}) holds, then%
\begin{equation}
\forall n,m\in\{1,...,N\} \text{ }\exists\phi_{m}^{n},\widehat{\sigma}%
^{n}:\forall s\in S_{ncu}\cap S^{n}\text{, }\phi_{m}^{n}(s)=\widehat{\sigma
}^{n}(s)\text{,}%
\end{equation}
which completes the proof.
\end{proof}

\section{Conclusion\label{sec04}}

Our aim was to identify the cases of SCAR games $\Gamma_{N}(G|s_{0}%
,\gamma,\varepsilon)$ where, the generally nonpositional trigger strategies
Nash equilibria $\overline{\sigma}$ \cite{Konstantinidis2019} are in fact
positional. The current, exhaustive study regarding the form and the values of
of $G,s_{0}$ and $N,\gamma,\varepsilon$ respectively showed that positional
$\overline{\sigma}$ profiles exist in exceptional cases, as reflected in the
following table.

\begin{center}
\begin{table}[h]%
\begin{tabular}
[c]{|c|c|c|c|c|}\hline
$\varepsilon$ & $G$ & $N$ & $s_{0}$ & $\gamma$\\\hline
\multicolumn{1}{|l|}{$\varepsilon\in(0,\frac{1}{N-1}]$} &
\multicolumn{1}{|l|}{$G$ is $\text{path }\mathcal{P}_{2}$} &
\multicolumn{1}{|l|}{$N=3$} & \multicolumn{1}{|l|}{$s_{0}\in S_{nc}$} &
\multicolumn{1}{|l|}{$\left[  \sqrt{\frac{\varepsilon}{1-\varepsilon}}%
,\frac{1}{2-2\varepsilon}\right]  $}\\\hline
\multicolumn{1}{|l|}{$\varepsilon=0$} & \multicolumn{1}{|l|}{$G$ is
$\text{path }\mathcal{P}_{2}$} & \multicolumn{1}{|l|}{$N>3$} &
\multicolumn{1}{|l|}{$s_{0}\in S_{nc}$} & \multicolumn{1}{|l|}{$\left(
0,\frac{1}{N-1}\right]  $}\\\hline
\multicolumn{1}{|l|}{$\varepsilon=0$} & \multicolumn{1}{|l|}{$G$ is
$\text{path }\mathcal{P}_{n}\text{ (}n\geq2\text{)}$} &
\multicolumn{1}{|l|}{$N\geq3$} & \multicolumn{1}{|l|}{$s_{0}\ $s.t.
c$\text{ops on one side of the robber}$} & \multicolumn{1}{|l|}{$\left(
0,\frac{1}{N-1}\right]  $}\\\hline
\multicolumn{1}{|l|}{$\varepsilon=0$} & \multicolumn{1}{|l|}{$G\in
\mathcal{G}_{2}$} & \multicolumn{1}{|l|}{$N\geq3$} &
\multicolumn{1}{|l|}{$s_{0}\in S_{nc}$} & \multicolumn{1}{|l|}{$\left(
0,1\right)  $}\\\hline
\multicolumn{1}{|l|}{$\varepsilon=0$} & \multicolumn{1}{|l|}{$G\in
\mathcal{G}_{3}$} & \multicolumn{1}{|l|}{$N\geq3$} &
\multicolumn{1}{|l|}{$s_{0}\in S_{nc}$} & \multicolumn{1}{|l|}{$\left(
0,\frac{1}{N-1}\right]  $}\\\hline
\end{tabular}
\par
%\textbf{Table 1}. Cases of positional $\overline{\sigma}$.
\caption{Cases of positional $\overline{\sigma}$.}%
\end{table}
\end{center}

Clearly paths and the values of $\varepsilon$ and $\gamma$ play a major role.
If $\varepsilon>0$ where cops in the SCAR game $\Gamma_{N}(G|s_{0}%
,\gamma,\varepsilon)$ have an incentive to cooperate, a positional
$\overline{\sigma}$ exists only in the exceptional (if not trivial) case where
two cops chase the robber on the path $\mathcal{P}_{2}$ and only for these
values of $\gamma$. In all remaining cases $\overline{\sigma}$ are
nonpositional. Also note that, games played on graphs in the class
$\mathcal{G}_{2}$ are the \emph{only} ones where a path is not involved, since
those in $\mathcal{G}_{3}$ can be seen as graphs in $\mathcal{G}_{2}$
connected to $\mathcal{P}_{2}$.

In addition to the above results, in the current study we have introduced the
\emph{state cop number }$c(G|s)$. This was crucial in our analysis for the
following reasons. 

\begin{enumerate}
\item Purely graph theoretical notions do not suffice because they do not take
into account the number of players. For example, if $G$ is the Petersen graph,
then: (i) if $N=3$, $G$ belongs to $\mathcal{G}_{2}\left(  N\right)  $ and all
$\Gamma_{N}\left(  G|s_{0},\gamma,\varepsilon\right)  $ possess positional
profiles $\overline{\sigma}$ (see Proposition \ref{Existposc(G)>N-1}),
whereas  (ii) if $N=4$,  all $\Gamma_{N}\left(  G|s_{0},\gamma,\varepsilon
\right)  $ possess only nonpositional profiles $\overline{\sigma}$ (see
Proposition \ref{prop0308}). 

\item The same holds for the classical cop number $c(G)$. For example, the
study of $\mathcal{G(N)}$ (i.e., the class of graphs with $c(G)>N-1$)\ in
Section \ref{sec040302}, required a finer subdivision into six subclasses
defined in terms of $c(G|s)$.
\end{enumerate}

In a future paper we intend to present a refined analysis of the $N$-player
SCAR game, concentrating  on subgame perfect equilibria (SPE) rather tnan NE.
To this end, three elements of the current paper will prove especially
useful.\ First of all, we will use the state cop number $c\left(  G|s\right)
$, because finding SPE requires analysis of optimal play in every subgame
$\Gamma_{N}(G|s,\gamma,\varepsilon)$, i.e., for every $s$ and  $c(G|s)$ gives
important information for this. Furthermore, the subdivision of class
$\mathcal{G(N)}$ and the analysis of zero-sum games $\Gamma_{N}^{m}$ will also
be useful.

\appendix

\section{Appendix\label{chapA}}

Here we introduce the \emph{state }$s$ \emph{cop number} $c_{N}\left(
G|s\right)  $ (or simply $c\left(  G|s\right)  $, when $N$ is taken for
granted, or is implied by the context). The notions presented here concern
exclusively  \emph{capturability} and thus they depend only\emph{ }on the
graph $G$ and the state $s$ to which they refer (and thus on the number of
players, their locations and whose turn is to move). Payoffs and initial
states play no role. This motivates us to define a \emph{CR-pregame\footnote{A
similar situation occurs with \emph{extensive game forms with perfect
information}, that is, structures of extensive games where players'
preferences are not specified \cite[p.90]{Rubinstein94} }} $\breve{\Gamma
}\left(  G\right)  $ consisting of a graph $G$, $N-1$ cop tokens and one
robber token, where move and capture rules are the same as in SCAR$\ \Gamma
_{N}(G|s_{0},\gamma,\varepsilon)$, but no payoffs or initial state are
specified. These notions apply also to the games $\Gamma_{N}^{m}%
(G|s_{0},\gamma,\varepsilon)$, including modified CR and to any other game
sharing this basic structure, since, for any initial state $s$ and strategy
profile $\sigma$, the infinite history produced is the same in all of them.

\medskip

First we define $s$\emph{-guaranteed capture profiles of }$k$\emph{-th order},
i.e. $k$-cops profiles which, when used in $\breve{\Gamma}_{N}\left(
G\right)  $, guarantee capture from state $s$, no matter how the rest $N-k$
players (including $R$) play.

\begin{definition}
A $k$-cops profile $\sigma=\left(  \sigma^{i_{1}},...,\sigma^{i_{k}}\right)  $
$\left(  k\in\left\{  1,...,N-1\right\}  \right)  $ in $\breve{\Gamma}%
_{N}\left(  G\right)  $ is called $s$\emph{-guaranteed capture profile}
($s$-gcp) \emph{of }$k$\emph{-th order\ }iff, starting from $s$, the profile
$(\sigma,\sigma^{\prime})$ leads to capture for all strategy profiles
$\sigma^{\prime}=\left(  \sigma^{j_{1}},...,\sigma^{j_{N-k}}\right)  $ with
$\left\{  j_{1},...,j_{N-k}\right\}  =\left\{  1,...,N\right\}  \backslash
\left\{  i_{1},...,i_{k}\right\}  $ of the rest players.
\end{definition}

Note that the definition of an $s$-gcp $\sigma$ implies guaranteed capture by
\emph{some} (one or more) cops, but \emph{not necessarily by one of the cops
involved in }$\sigma$.

And now we define the state cop number $c\left(  G|s\right)  $ in
$\breve{\Gamma}_{N}(G)$.

\begin{definition}
Consider the pregame $\breve{\Gamma}_{N}(G)$ ($N\geq2$) and $s\in S_{nc}$.

\begin{enumerate}
\item If a $k$-th order ($k\in\left\{  1,...,N-1\right\}  $) $s$-gcp exists,
then the \emph{state} $s$ \emph{cop number in }$\breve{\Gamma}_{N}\left(
G\right)  $ is denoted by $c\left(  G|s\right)  $ and defined to be the
minimum $k$ for which such a $k$-th order $s$-gcp exists.

\item Otherwise $c\left(  G|s\right)  =\infty$.
\end{enumerate}
\end{definition}

\begin{example}
\label{ExA}\normalfont Consider $\breve{\Gamma}_{3}(G)$ on graph $G$ depicted
in both Fig. \ref{fig0307}.a and Fig. \ref{fig0307}.b; note that $c(G)=2$.
\begin{figure}[ptbh]
\begin{center}
\begin{tikzpicture} \SetGraphUnit{2}
\Vertex[x= 0,y= 0]{1}
\Vertex[x=-1,y=-1]{2}
\Vertex[x= 1,y=-1]{3}
\Vertex[x= 0,y=-2]{4}
\Vertex[x= 2.5,y=-1]{5}
\Vertex[x= 4,y=-1]{6}
\Vertex[x= 5.5,y=-1]{7}
\node(A) [label=$C_1$] at (1,-0.7) {};
\node(B) [label=$R$  ] at (2.5,-0.7) {};
\node(C) [label=$C_2$] at (5.5,-0.7) {};
\Edge(1)(2)
\Edge(1)(3)
\Edge(2)(4)
\Edge(3)(4)
\Edge(3)(5)
\Edge(5)(6)
\Edge(6)(7)
\SetVertexNoLabel
\end{tikzpicture}
\qquad\begin{tikzpicture}
\SetGraphUnit{2}
\Vertex[x= 0,y= 0]{1}
\Vertex[x=-1,y=-1]{2}
\Vertex[x= 1,y=-1]{3}
\Vertex[x= 0,y=-2]{4}
\Vertex[x= 2.5,y=-1]{5}
\Vertex[x= 4,y=-1]{6}
\Vertex[x= 5.5,y=-1]{7}
\node(A) [label=$C_1$] at (4,-0.7) {};
\node(B) [label=$R$  ] at (1,-0.7) {};
\node(C) [label=$C_2$] at (2.5,-0.7) {};
\Edge(1)(2)
\Edge(1)(3)
\Edge(2)(4)
\Edge(3)(4)
\Edge(3)(5)
\Edge(5)(6)
\Edge(6)(7)
\SetVertexNoLabel
\end{tikzpicture}
\end{center}
\caption{(a) $s_{1}=\left(  3,7,5,3\right)  $, (b) $s_{2}=\left(
6,5,3,3\right)  $}%
\label{fig0307}%
\end{figure}Suppose the current state is $s_{1}=(3,7,5,3)$ as depicted in Fig.
\ref{fig0307}.a with $R$ having the move. It is clear that, starting from
$s_{1}$, $C_{1}$ can always ensure capture by going towards $R$ and thus
either effecting capture himself, or forcing $R$ to run into $C_{2}$. Hence
$c(G|s_{1})=1$.\footnote{Note however that $C_{1},R$ can move so that $C_{2}$
cannot effect capture.} From $s_{2}=(6,5,3,3)$ however (Fig. \ref{fig0307}.b)
with $R$ having again the move, both cops are needed to ensure capture and
thus $c(G|s_{2})=2$.
\end{example}

Finally we present the following theorem which establishes the connection
between $c\left(  G|s\right)  $ and the classic cop number $c(G)$. 

\begin{theorem}
\label{prop0315}Consider pregame $\breve{\Gamma}_{N}(G)$ with $N\geq2$. Then%
\begin{align}
c(G)  &  =K\leq N-1\Leftrightarrow\underset{s\in S_{nc}}{\max}c\left(
G|s\right)  =K\leq N-1\label{eq0H1}\\
c(G)  &  >N-1\Leftrightarrow\underset{s\in S_{nc}}{\max}c(G|s)=\infty.
\label{eq0H2}%
\end{align}

\end{theorem}

\begin{proof}
The complete proof will be given in a forthcoming paper. Here we give only the
(short and plausible) proof of (\ref{eq0H2}) which is what we will use in this
paper. Assume $c(G)>N-1$. This means there exists $s^{\prime}$ in
$\breve{\Gamma}_{N}(G)$ where the $N-1$ available cops do not suffice to
ensure capture. Then by definition of $c(G|s)$ it is $c(G|s^{\prime})=\infty$
and thus $\max_{s\in S_{nc}}c(G|s)=\infty$. Conversely, $\max_{s\in S_{nc}%
}c(G|s)=\infty$ implies there exists $s^{\prime}$ such that $c(G|s^{\prime
})=\infty$ and hence $N-1$ cops do not suffice to ensure capture in
$\breve{\Gamma}_{N}(G)$ from state $s^{\prime}$. Thus $c(G)>N-1$.
\end{proof}

\end{document}